\documentclass[11pt]{amsart}
\usepackage{latexsym}
\usepackage{amsmath,amssymb,amsfonts}
\usepackage{graphics}
\usepackage[all]{xy}

\def\eea{\end{eqnarray*}}

\newtheorem{thm}{Theorem}[section]
\newtheorem{prop}[thm]{Proposition}
\newtheorem{cor}[thm]{Corollary}
\newtheorem{lemma}[thm]{Lemma}

\newtheorem{remark}[thm]{Remark}

\begin{document}

\title[4-d gradient Ricci solitons with harmonic Weyl curvature]{On a classification of 4-d gradient Ricci solitons with harmonic Weyl curvature}

\author{Jongsu Kim}

\date{\today}

\address{Dept. of Mathematics, Sogang University, Seoul, Korea}
\email{jskim@sogang.ac.kr}

\thanks{This work was supported by the National Research Foundation of Korea(NRF) grant funded by the Korea government(MOE) (No.NRF-2010-0011704)}



\begin{abstract}
We study a characterization of 4-dimensional (not necessarily complete) gradient Ricci solitons $(M, g, f)$ which have harmonic Weyl curvature, i.e. $\delta W=0$. Roughly speaking, we prove that the soliton metric $g$  is locally isometric to  one of the following four types: an Einstein metric, the product $ \mathbb{R}^2 \times N_{\lambda}$ of the Euclidean metric and a 2-d Riemannian manifold of constant curvature ${\lambda} \neq  0$, a certain singular metric and a locally conformally flat metric.
The method here is motivated by Cao-Chen's works \cite{CC1, CC2} and  Derdzi\'{n}ski's study on Codazzi tensors \cite{De}.

Combined with the previous results on locally conformally flat solitons, our characterization yields a new classification of 4-d complete steady solitons with $\delta W=0$. For shrinking case, it reproves the rigidity result \cite{FG, MS} in 4-d.  It also helps to understand the expanding case; we now understand all 4-d non-conformally-flat ones with $\delta W=0$.
We also characterize {\it locally} 4-d (not necessarily complete) gradient Ricci solitons with harmonic curvature.
\end{abstract}

\maketitle

\setcounter{section}{0}
\setcounter{equation}{0}

\section{Introduction}

  A gradient Ricci soliton consists of a Riemannian manifold $(M, g)$ and a smooth function $f$ satisfying
$\nabla d f =  -Rc + \lambda g $, where $Rc$ denotes the Ricci tensor of $g$ and $\lambda$ is a constant. Gradient Ricci solitons are essential in Hamilton's Ricci flow theory as
singularity models of the flow. So it is important to understand their geometry and classify them.
 A gradient Ricci soliton is said to be shrinking, steady or expanding if $\lambda$ is positive, zero or negative, respectively.

\smallskip

Two dimensional gradient Ricci solitons are well understood; see \cite{BM} and references therein.
Any 3-d complete noncompact non-flat shrinker (shrinking Ricci soliton) is proved to be a quotient of the round cylinder $\mathbb{S}^2 \times \mathbb{R}$ in \cite{CCZ}; see also \cite{Iv3, NW, P}.
  For the 3-d  gradient steadiers (steady Ricci solitons), one may refer to \cite{Br, Cao1} and references therein.

\smallskip
In higher dimension, there are numerous rigidity and classification results under various geometric conditions. For the relevance to the current work, we shall focus on locally conformally flat solitons and its generalizations.

Complete locally conformally flat gradient shrinkers are classified to be a finite quotient of $\mathbb{R}^n$, $\mathbb{S}^n$, or $\mathbb{S}^{n-1} \times \mathbb{R}$, $n \geq  4$,
in \cite{CWZ, PW2,Z}; see also \cite{ELM, NW}.
Complete locally conformally flat gradient steadiers are  classified to be either flat or isometric to the Bryant soliton \cite{CC1, CM}.
The 4-d half conformally flat steadiers and shrinkers are studied in \cite{CW}. More generally, Bach-flat shrinkers  are classified in \cite{CC2}  and Bach-flat  steadiers with positive Ricci curvature  in \cite{CCCMM}.

\smallskip
A gradient soliton is said to be rigid if it is isometric to a
quotient of $N \times \mathbb{R}^k$ where $N$ is an Einstein manifold and $f = \frac{\lambda}{2}
|x|^2$ on the Euclidean factor.
 Fern\'{a}ndez-L\'{o}pez and Garc\'{i}a-R\'{i}o  \cite{FG} showed that an n-dimensional compact Ricci soliton $(M, g)$ is rigid
if an only if it has harmonic Weyl tensor $W$.
Then Munteanu and Sesum \cite{MS} proved that any $n$-dimensional complete gradient shrinker with
harmonic Weyl tensor is rigid.
In \cite{WWW}, Wu, Wu and Wylie showed that 4-d complete gradient shrinker with   $\delta W^{+}=0$ is
either Einstein, or a finite quotient of
$\mathbb{S}^3 \times \mathbb{R}$, $\mathbb{S}^2 \times \mathbb{R}^2$ or $\mathbb{R}^4$.

\medskip
The purpose of this article is to study  4-dimensional gradient Ricci solitons $(M, g, f)$ which have harmonic Weyl curvature. This work is most related to the above-mentioned works on locally confomally flat solitons and
to \cite{ MS} on shrinking solitons with $\delta W=0$. The latter needs control on geometric decay of curvature and volume from shrinker condition, while the former resorts to the nonnegative curvedness of metrics for locally confomally flat shrinking or steady solitons, which is proved in \cite{Ch, Z}.

As our study includes steady and expanding solitons  with $\delta W=0$, we can use neither geometric decay nor nonnegative curvedness. This work takes a different approach and is inspired by Cao and Chen's works \cite{CC1, CC2} and Derdzi\'{n}ski's \cite{De}.  Note that the harmonicity of Weyl tensor provides a Codazzi tensor $Rc- \frac{R}{6}g$. Riemannian metrics with a Codazzi tensor which have more than two distinct eigenvalue functions of Ricci tensor have been little understood, see Chapter 16 of \cite{Be}. In this article, combining with the soliton condition we managed to analyze in detail
the Codazzi tensor with three and four distinct eigenvalues.

Our argument is mostly local and produces a {\it local} description of soliton metrics and potential functions.  So far we worked out only in four dimension, but we hope that our perspective might provide some way to understand higher dimensional case.

The main theorem of this paper is as follows;
\begin{thm} \label{local}
Any four dimensional (not necessarily complete) connected gradient Ricci soliton  $(M, g, f)$  with harmonic Weyl curvature is one of the following four types.

\medskip
{\rm (i)} $g$ is an Einstein metric with $f$ a constant function.

\medskip
{\rm (ii)} For each point $p \in M$,   there exists a neighborhood $V$ of $p$
such that
 $(V, g)$ is isometric to a domain in the product $ \mathbb{R}^2 \times N_{\lambda}$
where $ \mathbb{R}^2$ has the Euclidean metric and $N_{\lambda}$ is a 2-dimensional Riemannian manifold of constant curvature ${\lambda} \neq  0$.   And $f = \frac{\lambda}{2} |x|^2$  modulo a constant on the Euclidean factor.

\medskip
{\rm (iii)} For each point $p \in M$,   there exists a neighborhood $V$ of $p$ with coordinates $(s,t, x_3, x_4)$ such that
$(V,g)$ is isometric to a domain in $\mathbb{R}^4 \setminus \{ s=0 \}$ with the Riemannian metric
 $ds^2 + s^{\frac{2}{3}} dt^2+ s^{\frac{4}{3}} \tilde{g}$,  where $ \tilde{g}$ is the Euclidean metric on the $(x_3, x_4)$-plane. Also,  $\lambda=0$ and $f=\frac{2}{3} \ln (s)$ modulo a constant.

\medskip

{\rm (iv)}
For each point $p$ in an open dense subset of $M$, there exists a neighborhood $V$ of $p$ with coordinates $(s,t, x_3, x_4)$ such that
$(V,g)$ is isometric to a domain in $\mathbb{R} \times W^3$ with the warped product metric $ ds^2 +    h(s)^2 \tilde{g},$
where $\tilde{g}$ is a constant curvature metric on a 3-manifold $W^3$ and $f$ is not constant.  And $g$ is locally conformally flat.

\end{thm}

\bigskip

For 4-d {\it complete} shrinking soliton case, we reprove the rigidity result in \cite{FG, MS} by a distinct method. For 4-d complete steady case, with the result of \cite{CC1, CM} on locally conformally flat solitons, we obtain the following classification.

\begin{thm} \label{steady}
A 4-dimensional complete steady gradient Ricci soliton with $\delta W=0$,
is either Ricci flat, or isometric
to the Bryant Soliton.
\end{thm}

 The  expanding solitons are much less rigid and many works have been done recently, e.g. \cite{PW, CD, SS, Cho} and references therein. We prove;

\begin{thm} \label{expand}
A  4-dimensional complete expanding gradient Ricci soliton  with harmonic Weyl curvature is one of the following;

\smallskip
{\rm (i)} $g$ is an Einstein metric with $f$ a constant function.

{\rm (ii)} $g$ is isometric to a finite quotient of $ \mathbb{R}^2 \times N_{\lambda}$
where $ \mathbb{R}^2$ has the Euclidean metric and $N_{\lambda}$ is a 2-dimensional Riemannian manifold of constant curvature ${\lambda} < 0$.   And $f = \frac{\lambda}{2} |x|^2$ on the Euclidean factor.

{\rm (iii)} $g$ is locally conformally flat.
\end{thm}

In \cite{PW} Petersen and Wylie proved that any complete gradient Ricci soliton with harmonic curvature is rigid.
But it is not clear if their argument extends to work for a local soliton.
The classification of any (not necessarily complete) gradient Ricci soliton with harmonic curvature comes from Theorem \ref{local}; we demonstrated it as Corollary \ref{har2} in the final section.

\medskip
To prove Theorem \ref{local},  from the harmonic Weyl curvature condition on gradient Ricci solitons, we observe by the arguments of \cite{CC2, FG} that  $\frac{\nabla f}{ | \nabla f | }$ is a Ricci-eigen vector field with its eigenvalue $\lambda_1$, there is a local function $s$ with $\nabla  s =\frac{\nabla f}{ | \nabla f | }$, and $\lambda_1$ and $R$ are functions of $s$ only.
Next we obtain important geometric informations on (Ricci-)eigenvalues, eigenvectors and eigenspaces  from the Codazzi tensor $Rc- \frac{R}{6}g$ through  Derdzi\'{n}ski's Lemma \ref{derdlem} and its extension Lemma \ref{abc60}.

Based on all the above, we show in Lemma \ref{raas} that
the Ricci-eigenvalues $\lambda_i$, $i=1,\cdots ,4$ locally depend only on the variable $s$; this key lemma is crucial in the later argument.
Then we divide the proof of Theorem \ref{local} into several cases, depending on the distinctiveness of  $\lambda_2, \lambda_3, \lambda_4$.
 There arise two subtle cases; when these three are pairwise distinct and when exactly two of them are equal. In the latter case we reduce the analysis to ordinary differential equations in Lemma \ref{bb8} and resolve them to get the types {\rm (ii)} and {\rm (iii)}.
In the former we compute on the soliton equation using Codazzi tensor property, which eliminates the case in Proposition \ref{4dformc}.

 The last case $\lambda_2=\lambda_3= \lambda_4$ is relatively simpler and produces the types {\rm (i)} and {\rm (iv)}.
Theorem \ref{steady}, \ref{expand} and Corollary \ref{har2} on the harmonic curvature case can be easily deduced from Theorem \ref{local}.

\bigskip
This paper is organized as follows. In section 2 we develop properties common to any gradient Ricci solitons with harmonic Weyl curvature and nonconstant $f$;  in particular we prove  that  $\lambda_i$'s, $i=1,\cdots ,4$  depend only on $s$.
In section 3 we study the case that  the  three $\lambda_i$'s, $i=2,3,4$, are pairwise distinct.
In section 4, 5 and 6, we analyze the case when two of the three $\lambda_i$'s, $i=2,3,4$, are equal.
In section 7, we treat the remaining case that $\lambda_2=\lambda_3= \lambda_4$.
In the final section 8, we summarize and prove theorems.

\section{Gradient Ricci solitons with harmonic Weyl curvature}

We shall begin by recalling some properties of a gradient Ricci soliton with  harmonic Weyl curvature  in a few lemmas.

 \begin{lemma} \label{solitonformulas}
For any gradient Ricci soliton $(M,g,f)$, we have;

\smallskip
{\rm (i)} $\frac{1}{2} dR = R(\nabla f, \cdot ) $, where $R$ in the left hand side denotes the scalar curvature, and $R(\cdot, \cdot)$ is a Ricci tensor.

\smallskip
{\rm (ii)} $R + |\nabla f|^2 - 2\lambda f = constant$.
 \end{lemma}

Our notational convention is as follows; for orthonormal vector fields $E_i$, $i=1, \cdots, n$ on an $n$-dimensional Riemannian manifold, the curvature components are

$R_{ijkl}:=R(E_i, E_j, E_k, E_l) = < \nabla_{E_i} \nabla_{E_j} E_k - \nabla_{E_j} \nabla_{E_i} E_k  -  \nabla_{[E_i, E_j]} E_k , E_l>  $.

\smallskip
\noindent We recall the formula (2.1) in  \cite{FG};

 \begin{lemma} \label{threesol}
For a  gradient Ricci soliton $(M^n, g, f)$ with harmonic Weyl curvature on an $n$-dimensional manifold $M^n$, we have;
\begin{eqnarray*} \label{solba}
R(X, Y, Z,\nabla  f ) &=& \frac{1}{n - 1} R(X,\nabla  f )g(Y, Z) - \frac{1}{n - 1} R(Y,\nabla f )g(X, Z)\\
&=& \frac{1}{2(n - 1)} dR(X) g(Y, Z) - \frac{1}{2(n - 1)} dR(Y)g(X, Z).
\end{eqnarray*}
\end{lemma}

One may mimic arguments in \cite{CC2} and get the next lemma.

 \begin{lemma} \label{threesolb}
Let  $(M^n, g, f)$ be  a  gradient Ricci soliton  with harmonic Weyl curvature. Let $c$ be a regular value of $f$ and $\Sigma_c= \{ x | f(x) =c  \}$  be the level surface of $f$. Then the following hold;

{\rm (i)} Where $\nabla f \neq 0$,  $E_1 := \frac{\nabla f }{|\nabla f  | }$ is an eigenvector field of $Rc$.

{\rm (ii)}  $R$ and $ |\nabla f|^2$  are constant on a connected component of $\Sigma_c$.

{\rm (iii)} There is a function $s$ locally defined with   $s(x) = \int  \frac{   d f}{|\nabla f|} $, so that

$ \ \ \ \ ds =\frac{   d f}{|\nabla f|}$ and $E_1 = \nabla s$.

{\rm (iv)}  $R({E_1, E_1})$ is constant on a connected component of $\Sigma_c$.

{\rm (v)}  Near a point in $\Sigma_c$, the metric $g$ can be written as

$\ \ \ g= ds^2 +  \sum_{i,j > 1} g_{ij}(s, x_2, \cdots  x_n) dx_i \otimes dx_j$, where
    $x_2, \cdots  x_n$ is a local coordinates system on $\Sigma_c$.

{\rm (vi)}  $\nabla_{E_1} E_1=0$.
\end{lemma}

\begin{proof}  Lemma \ref{threesol} gives $R(\nabla f, X) = 0$ for $X \perp \nabla f$, hence $E_1 = \frac{\nabla f }{|\nabla f  | }$ is an eigenvector of $Rc$.

 As $d R = 2 R(\nabla f, \cdot)$ from Lemma \ref{solitonformulas}, $d R (X) =0$ for $X \perp \nabla f$.
Also, $\frac{1}{2}\nabla_X |\nabla f|^2  = -  R( \nabla  f, X ) + \lambda g( \nabla  f, X ) = 0$  for $X \perp \nabla f$. We proved {\rm (ii)}.

 $d (\frac{   d f}{|\nabla f|}) = -\frac{1}{2 |\nabla f|^{\frac{3}{2}}} d |\nabla f|^{2} \wedge df= 0  $ as $\nabla_X ( |\nabla f|^2 )=0$ for $X \perp \nabla f$. So, {\rm (iii)} is proved.

Locally, $R$ may be considered as a function of the local variable $s$ only.
We can express $dR(E_1)=\frac{dR}{ds}ds (E_1)= \frac{dR}{ds} g( \nabla s, \nabla s ) =\frac{dR}{ds}$.
By Lemma \ref{solitonformulas},  we have $dR(E_1) = 2R({E_1, E_1}) |\nabla f|$, so $R({E_1, E_1})$ is constant on a connected component of $\Sigma_c$.

 As $\nabla f$ and the level surfaces of $f$ are perpendicular, one gets  {\rm (v)}.

For  {\rm (vi)}, one follows the proof of Proposition 5.1  in \cite{CC2}; with the local coordinates $s, x_2, \cdots  x_n$ in {\rm (v)}, one readily gets $\nabla s= \frac{\partial}{\partial s}$ so that  $[\frac{\partial}{\partial x_i}, \nabla s]=0 $. Then $\langle  \nabla s, \nabla s  \rangle=1$ and $\langle \frac{\partial}{\partial x_i}  , \nabla s \rangle=0$ yield {\rm (vi)}.
\end{proof}

A Codazzi tensor  on a Riemannian manifold $M$ is a symmetric tensor $A$ of covariant order 2 such that  $d^{\nabla} A =0$, which can be written in local coordinates as $\nabla_k A_{ij}  = \nabla_i A_{kj}$.
Derdzi\'{n}ski  \cite{De} described the following;
for a Codazzi tensor $A$ and a point $x$ in $M$, let  $E_A(x)$ be the number of distinct eigenvalues of $A_x$,
and set $M_A = \{    x \in M \  |  \  E_A {\rm \ is \ constant \ in \ a \ neighborhood of \ } x \}$, so that $M_A$ is an open dense subset of $M$ and that in each connected component of $M_A$, the eigenvalues are well-defined and differentiable functions.
The next lemma is from the section 2 of \cite{De}.

\begin{lemma} \label{derdlem}
For a Codazzi tensor $A$ on a Riemannian manifold $M$,
in each connected component of $M_A$,

{\rm (i)} Given distinct eigenfunctions $\lambda, \mu$ of $A$ and local vector fields $v, u$ such that  $A v = \lambda v$, $Au = \mu u$ with $|u|=1$, it holds that

$ \ \ \ \ \  v(\mu) = (\mu - \lambda) <\nabla_u u, v > $.

{\rm (ii)} For each eigenfunction $\lambda$, the $\lambda$-eigenspace distribution is integrable and its leaves are totally umbilic submanifolds of $M$.

{\rm (iii)} Eigenspaces of $A$ form mutually orthogonal differentiable distributions.
\end{lemma}

When a Riemannian manifold $M$ of dimension $n \geq 4$  has harmonic Weyl curvature, i.e. $\delta W=0$, it is equivalent to $d^{\nabla} (Rc - \frac{R}{2n-2} g) =0$.  So, $\mathcal{A}:=Rc - \frac{R}{2n-2} g$ is a Codazzi tensor.
By Lemma \ref{derdlem}, each eigenspace distribution of $\mathcal{A}$ is integrable in the open dense subset $M_{\mathcal{A}}$ of $M$.  The leaves are totally umbilic submanifolds of $M$.
Let
 $D_1, \cdots, D_k$ be all the eigenspace distributions of ${\mathcal{A}}$ in a connected component of $M_{\mathcal{A}}$. Then,
 the Ricci tensor also has $D_1, \cdots, D_k$ as its  eigenspace distributions. Let the dimension of $D_l$ be $d_l$ for $l=1, \cdots, k$.
 Then in a neighborhood of each point of the connected component of $M_{\mathcal{A}}$, there exist an orthonormal Ricci-eigen vector fields $E_i$, $i=1, \cdots  , n$ with corresponding eigenfunctions $\lambda_i$ such that
$E_1,  \cdots , E_{d_1}  \in D_1$,   $ \ \ E_{d_1 +1}   ,  \cdots, E_{d_1+ d_2}  \in D_2 $, $\cdots$ , and    $E_{d_1 + \cdots +d_{k-1}+1}   ,  \cdots, E_{n}  \in D_k $.

\bigskip
Let $(M^n,g,f)$  be a gradient Ricci soliton with harmonic Weyl curvature. As a gradient Ricci soliton, $(M,g,f)$ is real analytic in harmonic coordinates; see \cite{Iv}  or argue as in \cite[Prop. 2.4]{HPW}. Then if $f$ is not constant, $\{ \nabla f \neq 0  \}$ is open and dense in $M$.  As in the above paragraph, we consider orthonormal Ricci-eigen vector fields $E_i$ in a neighborhood of each point in  $M_A \cap \{ \nabla f \neq 0  \}$. By just requiring $E_1= \frac{\nabla f}{|\nabla f| }$ to be in $D_1$ and using Lemma \ref{threesolb}, we obtain;

\begin{lemma} \label{aa8}
Let $(M^n,g,f)$  be an $n$-dimensional gradient Ricci soliton with harmonic Weyl curvature and non constant $f$.
For any point $p$ in the open dense subset $M_{\mathcal{A}} \cap \{ \nabla f \neq 0  \}$ of $M^n$,
there is a neighborhood $U$ of $p$ where there exists an orthonormal Ricci-eigen vector fields $E_i$, $i=1, \cdots  , n$  such  that for all the eigenspace distributions $D_1, \cdots, D_k$ of ${\mathcal{A}}$ in $U$,

  \medskip

 {\rm (i)}  $E_1= \frac{\nabla f}{|\nabla f| }$ is in $D_1$,

 {\rm (ii)} for $i>1$, $E_i$ is tangent to smooth level hypersurfaces of $f$,

 {\rm (iii)} let $d_l$ be the dimension of $D_l$ for $l=1, \cdots, k$, then $E_1,  \cdots , E_{d_1}  \in D_1$,   $ \ \ E_{d_1 +1}   ,  \cdots, E_{d_1+ d_2}  \in D_2 $, $\cdots$ , and    $E_{d_1 + \cdots +d_{k-1}+1}   ,  \cdots, E_{n}  \in D_k $.
\end{lemma}

These local orthonormal Ricci-eigen vector fields $E_i$ of Lemma \ref{aa8} shall be called an {\it adapted frame field} of $(M, g, f)$.

\bigskip
For an adapted frame field $E_i$, $i=1, \cdots, n$, with $R_{ij}:= R(E_i, E_j)= \lambda_i \delta_{ij} $, from Lemma \ref{threesol}, for $j \in \{ 2, \cdots , n \}$ we get

\begin{equation} \label{ree}
R(E_1, E_j, E_j,\nabla  f )= \frac{1}{n - 1} Ric(E_1,\nabla  f )= \frac{1}{2(n - 1)} dR(E_1).
\end{equation}

Due to Lemma \ref{threesolb}, in a neighborhood of a point $p \in M_{\mathcal{A}} \cap \{ \nabla f \neq 0  \}$, $f$ and $R$ may be considered as functions of the variable $s$ only, and we write the derivative in $s$ by a prime: $f^{'} = \frac{df}{ds}$ and $R^{'} = \frac{dR}{ds}$, etc..
We recall $dR(E_1)=R^{'}ds (E_1)= R^{'} g( \nabla s, \nabla s ) =R^{'}$ and similarly  $df(E_1) = f^{'}$. Also, $df(E_1) = g(\nabla f, \frac{\nabla f}{ |\nabla f| })  = |\nabla f |$. So, $|\nabla f |= f^{'}$.
Then (\ref{ree}) becomes;
\begin{eqnarray} \label{0110g}
   R_{1jj1}  |\nabla  f| = \frac{1}{n - 1} R_{11}  |\nabla  f|  = \frac{1}{2(n - 1)} R^{'} .
  \end{eqnarray}

 \begin{lemma} \label{threesolb01}
For a  gradient Ricci soliton $(M,g,f)$ with harmonic Weyl curvature, and for a local adapted  frame field $\{ E_i \}$ in $M_{\mathcal{A}} \cap \{ \nabla f \neq 0  \}$,
setting
$ \zeta_i= - <   \nabla_{E_i}  E_i ,  E_1  >$, for $i >1$, we have;
\begin{equation} \label{lambda06a}
\nabla_{E_1}  E_1=0,    \ \ \  {\rm and}   \ \ \ \nabla_{E_i}  E_1 =   \frac{1}{ |\nabla f|} ( \lambda - \lambda_{i})E_i.
\end{equation}
\begin{equation} \label{lambda06}
\zeta_i =     \frac{1}{ |\nabla f|} (\lambda - \lambda_{i}).
\end{equation}
\end{lemma}

\begin{proof}
From Lemma \ref{threesolb} we get $\nabla_{E_1}  E_1=0$.
From the gradient Ricci soliton equation, for $i >1$,
$\nabla_{E_i}  E_1 = \nabla_{E_i} (\frac{\nabla f}{  | \nabla f |}) =   \frac{ \nabla_{E_i} \nabla f }{  | \nabla f |}=   \frac{ - R({E_i}, \cdot) + \lambda g( {E_i}, \cdot  ) }{  | \nabla f |}  =  - \frac{1}{ |\nabla f|} ( \lambda_{i} - \lambda)E_i$.
Then,   $    \zeta_i =- <   \nabla_{E_i}  E_i ,  E_1  > =   <    E_i ,   \nabla_{E_i} E_1  > =  \frac{1}{ |\nabla f|} (\lambda - \lambda_{i}) $.
\end{proof}

 \begin{lemma} \label{raas} For a 4-dimensional gradient Ricci soliton $(M, g, f)$ with harmonic Weyl curvature, and for a local adapted  frame field $\{ E_i \}$ in $M_{\mathcal{A}} \cap \{ \nabla f \neq 0  \}$,
  the Ricci-eigen functions
 $ \lambda_i $, $i=1, \cdots ,4$, are constant on a connected component of a regular level hypersurface $\Sigma_c$ of $f$, and so depend on the local variable  $s$ only. And $\zeta_i$, $i=2,3,4$, in Lemma \ref{threesolb01} also depend on $s$ only.
 In particular, we have
$E_i (\lambda_j) = E_i (\zeta_k)= 0$ for $i,k >1$ and any $j$.

 \end{lemma}

\begin{proof} We write $R_{ij}:= R(E_i, E_j)$. Recall that $\lambda_i=R_{ii}$.
We set $Rc^1 = Rc$ and for $k \geq 2$,
\noindent $Rc^k_{ij}= \sum_{s_1, s_2,  \cdots , s_{k-1}=1}^{4} R_{i s_1} R_{s_1 s_2}  \cdots  R_{s_{k-1} j}$ with its trace ${\rm tr}(Rc^k)= \sum_{i=1}^4 (\lambda_i)^k$.    We will show ${\rm tr}(Rc^k)$, $k=1,2,3$, depend on $s$ only.

First, $R={\rm tr}(Rc^1)$ and $\lambda_1= R_{11}$ depend on $s$ only  by Lemma  \ref{threesolb}. Next, for $k \geq 1$, writing  the Hessian $\nabla_j \nabla_i R :=  \nabla_{E_j} \nabla_{ E_{i}} R$, by Lemma \ref{threesolb01} we compute the following;
\begin{eqnarray} \label{3456}
 \sum_{j, s_1, s_2,  \cdots , s_{k-1}=1}^{4} ( \nabla_j \nabla_{ s_1} R) R_{s_1 s_2}  \cdots R_{s_{k-1} j}=  \sum_{j=1}^4 ( \nabla_j\nabla_{ j}R) (R_{j j})^{k-1}  \nonumber   \\
 = ( \nabla_1\nabla_{ 1}R) \lambda_1^{k-1}+ \sum_{i>1} (\nabla_i \nabla_i R) \lambda_i^{k-1} \hspace{2cm} \nonumber  \\
   =  (R^{''} ) \lambda_1^{k-1} +  \sum_{i>1} \{ E_iE_i(R) - (\nabla_{E_i} E_i) R \} \lambda_i^{k-1}  \hspace{0.3cm} \nonumber  \\
    = ( R^{''}) \lambda_1^{k-1}- \sum_{i>1}  \frac{R^{'}}{ |\nabla f|} (\lambda_i^{k} - \lambda \cdot \lambda_i^{k-1}) .  \hspace{1.5cm}
\end{eqnarray}
 In particular, for $k = 1$, (\ref{3456}) shows that
 $$\sum_{j=1}^4 \nabla_j \nabla_j R= R^{''} - \sum_{i>1}  \frac{R^{'}}{ |\nabla f|} (\lambda_{i}- \lambda) =R^{''} -   \frac{R^{'}}{ | \nabla f|}(R-  \lambda_{1}-3 \lambda), $$ which depends only on $s$. We drop summation symbols using the Einstein summation convention below.
\begin{eqnarray*}
\sum_{j=1}^4 \frac{1}{2}\nabla_j \nabla_{j} R=&  \nabla_j (f_i R_{ij}) =f_{ij} R_{ij} + f_i \nabla_j  R_{ij} = -(R_{ij} - \lambda g_{ij} ) R_{ij} + \frac{1}{2} f_i R_{i} \\
=&  -R_{ij}R_{ij} + \lambda R   + \frac{1}{2} f^{'} R^{'}. \hspace{5.5cm}
\end{eqnarray*}
So, ${\rm tr}(Rc^2)= R_{ij}R_{ij}$ depends only on $s$.

\medskip

We shall use the Codazzi equation  $\nabla_k R_{ij}  = \nabla_i R_{kj}  -  \frac{R_i}{6} g_{kj} + \frac{R_k}{6} g_{ij}$.
\begin{eqnarray} \label{rrt}
 \nabla_k (f_i R_{ij} R_{jk})  =  f_{ik} R_{ij} R_{jk}  +   f_i (\nabla_k R_{ij}) R_{jk} +   f_i R_{ij}  \nabla_k R_{jk}  \hspace{2.8cm} \nonumber \\
  = - (R_{ik} - \lambda g_{ik} ) R_{ij} R_{jk}  +   f_i (\nabla_i R_{kj}  -  \frac{R_i}{6} g_{kj} + \frac{R_k}{6} g_{ij}) R_{jk}
 +    \frac{1}{2} f_i R_{ij}  R_{j}  \nonumber \\
  = - {\rm tr}(Rc^3) + \lambda R_{ij} R_{ij}  +   \frac{1}{2} f_i \nabla_i ( R_{jk}R_{jk})  - f_i \frac{R_i}{6} R +\frac{ f_i R_k}{6} R_{ik}   \hspace{1.4cm} \nonumber \\
   +    \frac{1}{2} f_i R_{ij}  R_{j}.  \hspace{8.8cm}
\end{eqnarray}

All terms except ${\rm tr}(Rc^{3})$ in the right hand side of (\ref{rrt}) depend on $s$ only.
From (\ref{3456}) we also get
\begin{eqnarray} \label{rrtb}
2\nabla_k (f_i R_{ij} R_{jk})=\nabla_k ( R_j R_{jk}) = (\nabla_k  R_j) R_{jk}+ \frac{1}{2} R_j R_{j}  \hspace{1.5cm} \nonumber \\
 =  R^{''} R_{11} - \sum_{i>1}  \frac{R^{'}}{ |\nabla f|} (R_{ii}^2 - \lambda R_{ii})  +   \frac{1}{2} R_j R_j,  \hspace{0.5cm} \nonumber
\end{eqnarray}
 which depends only on $s$.
So, we compare this with (\ref{rrt}) to see that
 ${\rm tr}(Rc^3)$ depends only on $s$.
  Now $\lambda_1$ and  $\sum_{i=1}^4 (\lambda_i)^k$, $k=1, \cdots ,3$,  depend only on $s$.
This implies that
each $\lambda_i$, $i=1, \cdots ,4$,  is  a constant depending only on $s$.
By (\ref{lambda06}), $\zeta_i$, $i=2, 3 ,4$ depend on $s$ only.
\end{proof}

We now extend Lemma \ref{derdlem} (i);

\begin{lemma} \label{abc60} For a Riemannian metric $g$ of dimension $n \geq 4$ with harmonic Weyl curvature, consider orthonormal vector fields $E_i$, $i=1, \cdots n$ such that
$Rc(E_i, \cdot ) = \lambda_i g(E_i, \cdot)$. Then the following holds;

\smallskip
\noindent {\rm (i)}
 $(\lambda_j - \lambda_k ) \langle \nabla_{E_i} E_j, E_k \rangle   + \nabla_{E_i} \langle E_k, {\mathcal{A}}E_j \rangle=(\lambda_i - \lambda_k ) \langle \nabla_{E_j} E_i, E_k\rangle +\nabla_{E_j} \langle E_k, {\mathcal{A}} E_i \rangle, \ \ $

 for any $i,j,k =1, \cdots n$.

\smallskip
\noindent {\rm (ii)}  If $k \neq i$ and $k \neq j$,
$ \ \ (\lambda_j - \lambda_k ) \langle \nabla_{E_i} E_j, E_k\rangle=(\lambda_i - \lambda_k ) \langle \nabla_{E_j} E_i, E_k\rangle .$

\end{lemma}

\begin{proof}  The tensor ${\mathcal{A}}= Rc - \frac{R}{2n-2} g$ is a Codazzi tensor with
   eigenfunctions $\lambda_i -  \frac{R}{2n-2} $. We have
\begin{eqnarray*}
\langle (\nabla_{E_i}{\mathcal{A}}) E_j    , E_k \rangle =  - \langle \nabla_{E_i} E_j, {\mathcal{A}} E_k\rangle  -   \langle \nabla_{E_i} E_k, {\mathcal{A}}E_j\rangle  + \nabla_{E_i}\langle  E_k, {\mathcal{A}}E_j\rangle \hspace{1.5cm} \\
 = - (\lambda_k -  \frac{R}{2n-2})\langle \nabla_{E_i} E_j, E_k\rangle  -  ( \lambda_j -  \frac{R}{2n-2} )\langle \nabla_{E_i} E_k, E_j\rangle  + \nabla_{E_i}\langle  E_k, {\mathcal{A}}E_j\rangle \\
 = (\lambda_j - \lambda_k ) \langle \nabla_{E_i} E_j, E_k\rangle   + \nabla_{E_i}\langle  E_k, {\mathcal{A}}E_j\rangle.\hspace{5.7cm}
\end{eqnarray*}
As ${\mathcal{A}}$ is a Codazzi tensor,
$\langle(\nabla_{E_i}{\mathcal{A}}) E_j    , E_k\rangle =\langle(\nabla_{E_j}{\mathcal{A}}) E_i    , E_k\rangle$. So, we get {\rm (i)}.
Then  {\rm (ii)} holds since  $\nabla_{E_i} \langle  E_k, {\mathcal{A}}E_j \rangle =\nabla_{E_j} \langle E_k, {\mathcal{A}}E_i \rangle =0$.
\end{proof}

\begin{lemma} \label{abc60b} For a  gradient Ricci soliton  $(M,g,f)$ with harmonic Weyl curvature, and for a local adapted frame field $\{ E_i \}$ in $M_{\mathcal{A}} \cap \{ \nabla f \neq 0  \}$, the following holds.

\medskip
For $i,j,k >1$, with $k \neq i$ and $k \neq j$,  setting $\Gamma^k_{ij}:= < \nabla_{E_i} E_j,  E_k>$,

$ (\zeta_k - \zeta_j ) \Gamma^k_{ij}=(\zeta_k - \zeta_i ) \Gamma^k_{ji}, \ \  $  $(\zeta_k - \zeta_j ) \Gamma^k_{ij}=(\zeta_i - \zeta_j ) \Gamma^i_{kj} \ $
and
$ \  \    \Gamma^k_{ij} =  - \Gamma^j_{ik}$.
\end{lemma}

\begin{proof}
From (\ref{lambda06}) and Lemma \ref{abc60}, $(\zeta_k - \zeta_j ) \Gamma^k_{ij}=(\zeta_k - \zeta_i ) \Gamma^k_{ji}$.  Others hold readily.
\end{proof}

\section{4-dimensional solitons with distinct  $\lambda_2, \lambda_3, \lambda_4$}

Let $(M,g,f)$ be a four dimensional gradient Ricci soliton with harmonic Weyl curvature and non constant $f$.
In a neighborhood of any point in the open dense subset $M_{\mathcal{A}} \cap \{ \nabla f \neq 0  \}$ of $M$, there exists
an adapted frame field $E_j$, $j=1,2,3,4$, of Lemma \ref{aa8} with its eigenfunction $\lambda_j$

We may only consider three cases depending on the distinctiveness of $\lambda_2,\lambda_3,\lambda_4$;
the first case is when  $ \lambda_i$, $i=2,3,4$ are all equal (on an open subset),
 and  the second  is when exactly two of the three are equal. And the last is when the three  $ \lambda_i$, $i=2,3,4$, are mutually different.

In this section we shall study the last case.
\begin{lemma}\label{77}
Let $(M,g,f)$ be a four dimensional gradient Ricci soliton with harmonic Weyl curvature and non constant $f$.
Suppose that for an adapted frame fields $E_j$, $j=1,2,3,4$,
in an open subset $W$ of $M_{\mathcal{A}} \cap \{ \nabla f \neq 0  \}$,   the eigenfunctions $\lambda_2, \lambda_3, \lambda_4$ are distinct from each other. Then the following hold in $W$;

\medskip
\noindent for $i, j >1$, $i \neq j$,

$\nabla_{E_1}  E_1 =0$, $\ \ \ \ \ \nabla_{E_i} E_1 = \zeta_i E_i $,
$\ \ \ \ \ \  \nabla_{E_i}  E_i = -\zeta_i E_1  $,  $\ \ \ \ \ \nabla_{E_1} E_i=0$.

 $\nabla_{E_i} E_j= \Gamma_{ij}^k E_k$ where $k \neq 1,i,j. \ \ \ $

\end{lemma}

\begin{proof}
From Lemma \ref{threesolb01} we have  $\nabla_{E_1}  E_1 =0$ and $\nabla_{E_i} E_1 = \zeta_i E_i$. From Lemma  \ref{derdlem} (i) and Lemma \ref{raas},
 $ \langle  \nabla_{E_i} E_i ,   E_j\rangle=0$.
  And  $ \langle  \nabla_{E_i} E_i ,   E_1\rangle= - \langle   E_i ,   \nabla_{E_i} E_1\rangle= - \zeta_i$. So, we get  $\nabla_{E_i}  E_i = -\zeta_i E_1  $.
Now, $ -\langle  \nabla_{E_i} E_j, E_i \rangle=0$,
 $\langle  \nabla_{E_i} E_j, E_j \rangle=0  $.  And $\langle\nabla_{E_i} E_j, E_1 \rangle = -\langle   \nabla_{E_i} E_1 ,  E_j\rangle =0 $. So,  $\nabla_{E_i} E_j= \Gamma_{ij}^k E_k$ where $k \neq 1,i,j$.  Clearly  $\Gamma_{ij}^k =- \Gamma_{ik}^j $.

From Lemma  \ref{abc60} {\rm (ii)}, $(\lambda_i - \lambda_j ) \langle \nabla_{E_1} E_i, E_j\rangle=(\lambda_1 - \lambda_j ) \langle \nabla_{E_i} E_1, E_j\rangle $.  As $\langle \nabla_{E_i} E_1, E_j\rangle=0 $,  $\langle\nabla_{E_1} E_i, E_j\rangle=0$.     This gives $\nabla_{E_1} E_i =0$.
\end{proof}

\noindent
From above Lemma, we may write
\begin{equation} \label{coeff01}
[E_2, E_3] = \alpha E_4, \ \ \  [E_3, E_4] = \beta E_2,  \ \ \ [E_4, E_2] = \gamma E_3.
\end{equation}

\begin{lemma} \label{4dformb} Under the hypothesis of Lemma \ref{77}, we have the following relation on $\zeta_i$'s and the coefficients of {\rm (\ref{coeff01})}.
\begin{eqnarray*}
E_1(\alpha) = \alpha ( \zeta_4 -  \zeta_2 - \zeta_3    ),  \ \  \    E_1(\beta) = \beta ( \zeta_2 - \zeta_3 - \zeta_4    ), \  \      E_1(\gamma) = \gamma ( \zeta_3 -  \zeta_2 - \zeta_4    ) \\
\beta =   \frac{(\zeta_3 - \zeta_4   )^2}{(\zeta_2 - \zeta_3   )^2} \alpha,   \ \ \ \ \ \ \   \ \ \ \ \  \gamma =   \frac{(\zeta_2 - \zeta_4   )^2}{(\zeta_2 - \zeta_3   )^2} \alpha. \hspace{5.5cm}
\end{eqnarray*}
\end{lemma}
\begin{proof}
From Jacobi identity $[[X, Y], Z]   +   [[Y, Z], X] + [[Z, X], Y]=0  $ applied to $(X, Y, Z) = (E_1, E_2, E_3)$ gives
$E_1(\alpha) = \alpha ( \zeta_4 - \zeta_2 - \zeta_3    )$. Apply it to  $E_1, E_2, E_4$ and  $E_1, E_3, E_4$, we get  the next two.

Using $2 \langle  \nabla_X Y, Z \rangle  =  X \langle Y, Z  \rangle   + Y \langle  X, Z \rangle  - Z  \langle  X,Y \rangle    +\langle [X,Y], Z \rangle  -  \langle [X,Z], Y \rangle -  \langle [Y,Z], X \rangle  $ for vector fields $X,Y, Z$, from Lemma \ref{abc60b} we get;

$ \frac{-\alpha - \gamma + \beta}{2}= \Gamma_{24}^3 =  \frac{( \zeta_2 - \zeta_4 )}{ \zeta_3 - \zeta_4   }\Gamma_{34}^2    =\frac{( \zeta_2 - \zeta_4 )}{ \zeta_3 - \zeta_4   }\frac{\alpha - \gamma + \beta}{2}. \ \  $
 So,
$-\alpha - \gamma + \beta= \frac{( \zeta_2 - \zeta_4 )}{ \zeta_3 - \zeta_4   }(\alpha - \gamma + \beta).$
By symmetry we have, $-\beta - \alpha + \gamma= \frac{( \zeta_3 - \zeta_2 )}{ \zeta_4 - \zeta_2   }(\beta - \alpha + \gamma) $
and $-\gamma - \beta + \alpha= \frac{( \zeta_4 - \zeta_3 )}{ \zeta_2 - \zeta_3   }(\gamma- \beta + \alpha).$
From these, we can get the other formulas.
\end{proof}

\begin{lemma} \label{4dform}
Let a four dimensional gradient Ricci soliton $(M,g,f)$ with harmonic Weyl curvature satisfy the hypothesis of Lemma \ref{77}. Then the following hold in $W$;

\medskip
\noindent For distinct $\ i, j, k>1$,
$ \  R_{1ii1} =-\zeta_i^{'}  -  \zeta_i^2 =R_{1jj1}, $ where $\zeta_i^{'}= \frac{d \zeta_i}{ds}$,  $\ \     R_{1ij1}= 0$.

$R_{11} =  -3\zeta_2^{'}
 -  3\zeta_2^2  .$

$R_{22} = - \zeta_2^{'}
 -  \zeta_2^2 -\zeta_2 \zeta_3
 -\zeta_2 \zeta_4  -2 \Gamma_{34}^2 \Gamma_{43}^2 .$

$R_{33} = -\zeta_3^{'}
 -  \zeta_3^2 -\zeta_3 \zeta_2
 -\zeta_3 \zeta_4  +2 \frac{( \zeta_2 - \zeta_4 )}{ \zeta_3 - \zeta_4   }\Gamma_{34}^2 \Gamma_{43}^2 .$

$R_{44} = -\zeta_4^{'}
 -  \zeta_4^2 -\zeta_4 \zeta_2
 -\zeta_4 \zeta_3  +2\frac{( \zeta_2 - \zeta_3 )}{ \zeta_4 - \zeta_3   } \Gamma_{34}^2 \Gamma_{43}^2 .$

$R_{1i}=0,$  $\ \ \ \ \ \ \ R_{ij} = E_k (\Gamma^k_{ij}).$

\end{lemma}
\begin{proof} One uses Lemma \ref{77} and Lemma \ref{raas}. Recall $ R_{1ii1}  =R_{1jj1} $ from (\ref{0110g}).
By direct computation we get $ R_{1ii1} =-\zeta_i^{'}  -  \zeta_i^2$, $R_{jiij} = -\zeta_j \zeta_i  - \Gamma_{ji}^k \Gamma_{ik}^j  - \Gamma_{ji}^k \Gamma_{ki}^j + \Gamma_{ij}^k \Gamma_{ki}^j$ and $ R_{kijk} = E_k (\Gamma^k_{ij})$.
Use  Lemma \ref{abc60b} to express $R_{33}$ and $R_{44}$.

\end{proof}

\noindent  Here we set $a:= \zeta_2$, $b:= \zeta_3$ and  $c:= \zeta_4$. From the soliton equation $\lambda- \zeta_i f^{'} = R_{ii}$, $i >1$ and Lemma \ref{4dform},

$ -(a - b)f^{'} =R_{22} - R_{33} =  (b  -a )c
 -2\{ 1+  \frac{( a - c )}{ b - c   }\}\Gamma_{34}^2 \Gamma_{43}^2   $.
So,
\begin{equation} \label{fpri}
 f^{'} = c  +2 \frac{( a+ b   - 2c )}{ (a - b)(b - c )  }\Gamma_{34}^2 \Gamma_{43}^2.
 \end{equation}
Similarly,
 $-(a - c)f^{'}=  (c  -a )b
 -2\{ 1+  \frac{( a - b )}{ c - b   }\}\Gamma_{34}^2 \Gamma_{43}^2   $.
 So,
\begin{equation} \label{fpr}
f^{'}=  b
 +2  \frac{( a + c - 2b )}{(a - c) (c - b )  }\Gamma_{34}^2 \Gamma_{43}^2.
 \end{equation}
From (\ref{fpri}) and (\ref{fpr}),  we get
\begin{equation} \label{fpr31}
4\Gamma_{34}^2 \Gamma_{43}^2 = \frac{(a - b)(a - c) (b - c )^2}{ ( a^2  +b^2+ c^2 - ab  -  bc - ac ) },
 \end{equation}

\begin{equation} \label{fpr2}
f^{'}=   \frac{a^2  b +   a^2 c  + a   b^2 +  a  c^2 +  b^2 c +   c^2   b   - 6a  b   c }{ 2 ( a^2  +b^2+ c^2 - ab  -  bc - ac ) }.
 \end{equation}

We are now ready to prove the following.

\begin{prop} \label{4dformc}
Let $(M,g,f)$ be a four dimensional gradient Ricci soliton with harmonic Weyl curvature and non constant $f$.
For any adapted frame field $E_j$, $j=1,2,3,4$,
in an open dense subset $M_{\mathcal{A}} \cap \{ \nabla f \neq 0  \}$ of $M$,
 the three eigenfunctions $\lambda_2, \lambda_3, \lambda_4$ cannot be  pairwise distinct, i.e. at least two of the three coincide.

\end{prop}
\begin{proof}
Suppose that  $\lambda_2, \lambda_3, \lambda_4$ are pairwise distinct. We shall prove then that $g$ should be an Einstein metric, so a contradiction.

 \noindent In this proof again we set $a:= \zeta_2$, $b:= \zeta_3$ and  $c:= \zeta_4$.
From (\ref{fpr31}) and Lemma \ref{abc60b},

$ (\alpha - \gamma + \beta)^2= 4(\Gamma_{34}^2)^2  = 4 \Gamma_{34}^2 \Gamma_{43}^2 \frac{(a-b)}{(a-c)}  =\frac{(a - b)^2 (b - c )^2}{ ( a^2  +b^2+ c^2 - ab  -  bc - ac ) }$.

\noindent For convenience set $P:= a^2  +b^2+ c^2 - ab  -  bc - ac $.
From Lemma \ref{4dformb},
\begin{eqnarray*}
(\alpha - \gamma + \beta)^2 =\alpha^2 \{1 -  \frac{(a - c   )^2}{(a - b   )^2}  +  \frac{(b - c   )^2}{(a - b   )^2}  \}^2  = \frac{4 \alpha^2 (b  - c )^2}{(a - b   )^2}.
 \end{eqnarray*}

So,   $ \alpha^2    = \frac{(a - b)^4 }{ 4P } .$ Since $a,b,c$ are all functions of $s$ only, so is $\alpha$.

Differentiating this in $s$ and using $ b^{'}  - a^{'} =  a^2  -  b^2  $ and  $ c^{'} - a^{'} =  a^2   -  c^2  $,  we get
 \begin{eqnarray*}
 2 \alpha \alpha^{'}    =  &  \frac{(a - b)^3 (a^{'} - b^{'}) }{ P}  -   \frac{(a - b)^4 ( 2a a^{'}  +2b b^{'} + 2c c^{'}  - a b^{'} - b a^{'} -  a c^{'} -c a^{'} - c b^{'}- b c^{'}  )  }{ 4P^2 } \hspace{1.5cm} \\
 =  &  \frac{-(a - b)^3 (a^2 - b^2) }{ P}  -   \frac{(a - b)^4 \{ (a-b) (a^{'} -b^{'})  + (a-c) (a^{'} -c^{'})+ (b-c) (b^{'} -c^{'})    \}  }{ 4P^2 } \hspace{1.5cm} \\
  =  &  \frac{-(a - b)^4 (a+b) }{ P}  +   \frac{(a - b)^4 \{ (a-b) (a^2 -b^2)  + (a-c) (a^2 -c^2)+ (b-c) (b^2 -c^2)    \}  }{ 4P^2 } \hspace{1.5cm} \\
    =  &  -\frac{(a - b)^4  }{ P} [(a+b)  -   \frac{ \{ 2(a^3 +b^3 + c^3 -3abc) + 6abc -a^2b -ab^2 -a^2c -ac^2 -b^2c -bc^2 \}  }{ 4P }] \hspace{0.5cm} \\
 =  &  -\frac{(a - b)^4  }{ P} [(a+b)  - \frac{(a+b+c)}{2} - \frac{ \{ 6abc -a^2b -ab^2 -a^2c -ac^2 -b^2c -bc^2 \}  }{ 4P }] \hspace{1.5cm} \\
  =  &  -\frac{(a - b)^4  }{ P} [ \frac{(a+b-c)}{2} - \frac{ \{ 6abc -a^2b -ab^2 -a^2c -ac^2 -b^2c -bc^2 \}  }{ 4P }] \hspace{3cm}
 \end{eqnarray*}

Meanwhile, from Lemma \ref{4dformb} and  $ \alpha^2    = \frac{(a - b)^4 }{ 4P } $,

\begin{equation*}
 2 \alpha \alpha^{'}= 2 \alpha E_1(\alpha)=-2 \alpha^2 ( a +  b - c    )= -\frac{(a - b)^4 }{ 2P }( a +  b - c    ).
 \end{equation*}

\noindent Equating these two expressions for $ 2 \alpha \alpha^{'}$, we get;

$6abc = a^2b + b^2a + a^2c + c^2a + b^2c + c^2b   $.
From (\ref{fpr2}), $f^{'} =0$. So, $g$ is an Einstein metric.
\end{proof}

\section{4-dimensional soliton with $\lambda_2 \neq \lambda_3 =\lambda_4$ }

In this section we begin to study the case when exactly two of the three eigenvalues $\lambda_2, \lambda_3, \lambda_4$ are equal.  We may well assume that $ \lambda_2  \neq \lambda_3=  \lambda_4 $.

\begin{lemma} \label{claim112na}
Let $(M, g, f)$ be a four dimensional gradient Ricci soliton with harmonic Weyl curvature.
Suppose that  $ \lambda_2  \neq \lambda_3=  \lambda_4$ for an adapted frame fields $E_j$, $j=1,2,3,4$,
on an open subset of $M_{\mathcal{A}} \cap \{ \nabla f \neq 0  \}$.
Then the following hold on the open subset;

 \medskip
$\nabla_{E_1}  E_1=0  $.

$\nabla_{E_i}  E_1= \zeta_i(s) E_i$ for $i=2,3,4$, with $ \zeta_i(s)=\frac{ 1}{|\nabla f |} (\lambda -   \lambda_i ) $.

$\nabla_{E_2}  E_2 = -\zeta_2(s) E_1  $.
$\nabla_{E_3}  E_3 = -\zeta_3  E_1 -  \beta_3  E_4  $, $\nabla_{E_4}  E_4 = -\zeta_4  E_1 + \beta_4 E_3  $,

$   \ \     $ for some functions $\beta_3$ and $\beta_4$.

$\nabla_{E_1}  E_2=  0$, $\nabla_{E_1}  E_3= \rho E_4$ and $\nabla_{E_1}  E_4= -\rho E_3$  for some function $\rho$.

$\nabla_{E_2}  E_3=   q E_4$ and $\nabla_{E_2}  E_4=  - q E_3$  for some function $q$.

 $\nabla_{E_3}  E_2=0$ and $\nabla_{E_4}  E_2=0$.

$\nabla_{E_3}  E_4 = \beta_3 E_3  $ and $\nabla_{E_4}  E_3 = - \beta_4 E_4  $.

$[E_1, E_2]= -\zeta_2 E_2  $ and $[E_3, E_4]= \beta_3 E_3 + \beta_4 E_4.$

\medskip
In particular, the distribution spanned by $E_1$ and $E_2$ is integrable. So is that  spanned by $E_3$ and $E_4$.

\end{lemma}

\begin{proof}
The formula for  $\nabla_{E_i} E_1$, $i \geq 1$,  comes from (\ref{lambda06a}).

Then from Lemma \ref{raas} and Lemma \ref{derdlem} (i); $ (\lambda_2 - \lambda_i   )\langle\nabla_{E_2}  E_2, E_i\rangle =  E_i(\lambda_2) =0  $ for $i=3,4$ and
$\langle\nabla_{E_2}  E_2, E_1\rangle = - \langle  E_2,  \nabla_{E_2}  E_1 \rangle= -\zeta_2(s)$.
So, $\nabla_{E_2}  E_2 = -\zeta_2(s) E_1  $.
By similar argument,
$\nabla_{E_3}  E_3 = -\zeta_3  E_1 -  \beta_3  E_4  $, $\nabla_{E_4}  E_4 = -\zeta_4  E_1 + \beta_4 E_3  $, for some functions $\beta_3$ and $\beta_4$.

From Lemma \ref{abc60} {\rm (ii)},
$(\lambda_2 - \lambda_i ) \langle \nabla_{E_1} E_2, E_i\rangle=(\lambda_1 - \lambda_i ) \langle \nabla_{E_2} E_1, E_i\rangle=(\lambda_1 - \lambda_i) \langle\zeta_2 E_2, E_i\rangle=0 ,$ for $i=3,4$.
 So, $\langle\nabla_{E_1}  E_2, E_i\rangle =0$, for $i=3,4$.  As $\langle\nabla_{E_1}  E_2, E_1\rangle = -\langle  E_2, \nabla_{E_1} E_1\rangle=0$, we have   $\nabla_{E_1}  E_2=  0$.

As $\langle\nabla_{E_1}  E_3,  E_2\rangle = -\langle E_3,  \nabla_{E_1}  E_2\rangle=  0$, one can readily get $\nabla_{E_1}  E_3= \rho E_4$ for some function $\rho$  and $\nabla_{E_1}  E_4= -\rho E_3$. And $\nabla_{E_2}  E_3=   q E_4$ for some function $q$  and $\nabla_{E_2}  E_4=  - q E_3$.

From Lemma \ref{abc60} {\rm (ii)},
$(\lambda_2 - \lambda_4 ) \langle \nabla_{E_3} E_2, E_4\rangle=(\lambda_3 - \lambda_4 ) \langle \nabla_{E_2} E_3, E_4\rangle=0 .$
So,  $\langle\nabla_{E_3}  E_2, E_4\rangle  =0$.
As we have $\langle\nabla_{E_3}  E_2, E_a\rangle  =0$ for $i=1, 3$ from above,
we get
 $\nabla_{E_3}  E_2=0$.  Similarly, $\nabla_{E_4}  E_2=0$.

One can easily compute $\nabla_{E_3}  E_4 = \beta_3 E_3  $ and $\nabla_{E_4}  E_3 = - \beta_4 E_4  $.
From above we get $[E_1, E_2]= -\zeta_2 E_2$ and  $[E_3, E_4]= \beta_3 E_3 + \beta_4 E_4$.
\end{proof}

\begin{lemma} \label{claim112}
  Let $D^1$ and $D^2$ be both two dimensional smooth integrable distributions on a domain $\Omega$ of a four dimensional manifold that span the tangent space $T_p \Omega$ for each $p \in \Omega$. Let $p_0$ be a point in $\Omega$. Then there is a coordinate neighborhood $ (x_1,x_2,x_3,x_4)$ near $p_0$  so that  $D^1$ is tangent to the 2-dimensional level sets $ \{ (x_1, x_2, x_3,x_4) |  \ x_3, x_4 \ {\rm constants} \} $ and $D^2$ is tangent to the level sets $ \{ (x_1, x_2, x_3,x_4) |  \ x_1, x_2 \ {\rm constants} \} $.
\end{lemma}

 \begin{proof}
   By Frobenius theorem, there is a coordinate neighborhood $ {\bf x}:=(x,y,z,w) $ near  $p_0$  so that $D^2$ is tangent to the sets $ \{ (x,y,z,w) |  \ x,y \ {\rm constants} \} $. We may assume that $(x(p_0), y(p_0), w(p_0), z(p_0) )=  (0, 0,0 ,0)$.

Then there are two vector fields $v_1= (a_1, b_1, c_1, d_1):= a_1\frac{\partial}{\partial x} + b_1\frac{\partial}{\partial y} + c_1\frac{\partial}{\partial z}+ d_1\frac{\partial}{\partial w} $ and $v_2=(a_2, b_2, c_2, d_2)$ for points $p$ near $p_0$ in $D^1$,  with $(a_1(p), b_1(p))$ and $(a_2(p), b_2(p))$ being linearly independent as two dimensional vectors; if not, $D^1_{p}$ and $D^2_{p}$ won't span $T_{p} \Omega$.

By considering $X_1:= \alpha_1 v_1  +  \beta_1 v_2$ and  $X_2:= \alpha_2 v_1  +  \beta_2 v_2$ for smooth functions $\alpha_i, \beta_i$, we have smooth vector  fields  $X_1, X_2 \in D^1$,  of the form $X_1(p)= (1,0, a_1(p), a_2(p))  $ and $X_2 = (0,1, b_1(p), b_2(p))$ for $p$ near  $p_0$ with smooth functions $a_i, b_i$, $i=1,2$.

Consider the one-parameter subgroup $\phi_t$  of  $X_1$ and $\psi_s$ of $X_2$;
 $\frac{d }{dt} \phi_t(p) =  (1,0, a_1(\phi_t(p)), a_2(\phi_t(p))  )_{\phi_t(p)} $ and   $\frac{d }{ds} \psi_s(p)  = (0,1, b_1(\psi_s(p)), b_2(\psi_s(p)))$.

 Define a map $\Phi$ on a neighborhood of the origin in $\mathbb{R}^4= \{ (x_1, x_2, x_3, x_4) \}$ into $\mathbb{R}^4=  \{ (x,y,z,w) \}$ by
  $\Phi(x_1, x_2, x_3, x_4)  :=  \phi_{x_1} \psi_{x_2}  (0, 0, x_3, x_4)$. This $\Phi$  gives a local coord. system near  $p_0$.
From $\frac{d }{ds} \psi_s(p)  = (0,1, b_1(\psi_s(p)), b_2(\psi_s(p))) $, we get $\psi_{x_2}  (0, 0, x_3, x_4)=(0, x_2, *, *)$ and similarly
 $\phi_{x_1} \psi_{x_2}  (0, 0, x_3, x_4) = \phi_{x_1}(0, x_2, *, *)  =(x_1, x_2, *, *)$.

So, $\Phi (x_1, x_2, x_3, x_4) = (x_1, x_2, *, *)$.
  Then we get $\Phi_* ( \frac{\partial}{ \partial x_3}  ), \Phi_* ( \frac{\partial}{ \partial x_4}  )   \in  {\rm span}(\frac{\partial}{ \partial z} , \frac{\partial}{ \partial w}) =D^2 $. So,  $D^2$ is spanned by $\Phi_* ( \frac{\partial}{ \partial x_3}  )$ and  $\Phi_* ( \frac{\partial}{ \partial x_4}  ) $.

\medskip
As $D^1$ is integrable, in a neighborhood of each point  $q_0:= (0, 0, c, d)$ near the origin, there is a unique surface $S_{q_0}$ containing $q_0$ which is tangent to the distribution $D^1$ at each point of $S_{q_0}$. As $X_1$ and $X_2$ are vector fields on  $S_{q_0}$, at each point $q \in S_{q_0}$ we have $\{ \psi_{x_2} (q)  \  | \   x_2 \in (-\epsilon, \epsilon) \} \subset S_{q_0}$ and $\{ \phi_{x_1} (q) \  | \  x_1 \in (-\epsilon, \epsilon) \} \subset S_{q_0}$ for small $\epsilon$.
Therefore the set $\{  \phi_{x_1} \psi_{x_2}  (0, 0, c, d) \  |  \   x_1, x_2 \in (-\varepsilon, \varepsilon)    \}$, for small $\varepsilon$,
 coincides with  $S_{q_0}$ near $q_0$.
So, we get
$\Phi_* ( \frac{\partial}{ \partial x_1}  ),   \Phi_* ( \frac{\partial}{ \partial x_2}  )  \in  D^1 ,$
and $D^1$ is spanned by $\Phi_* ( \frac{\partial}{ \partial x_1}  ), \Phi_* ( \frac{\partial}{ \partial x_2}  ) $.
 Now we have obtained a new coordinates system $\Phi^{-1}  \circ {\bf x} $ with the desired property. This proves the lemma.
\end{proof}

Using Lemma \ref{claim112na} and
 Lemma \ref{claim112}, we can express the metric $g$ in the following lemma.

\begin{lemma} \label{claim112b}
Let $(M, g, f)$ be a four dimensional gradient Ricci soliton with harmonic Weyl curvature.
Suppose that  $ \lambda_2  \neq \lambda_3=  \lambda_4$ for an adapted frame fields $E_j$, $j=1,2,3,4$,  on an open subset $U$ of $M_{\mathcal{A}} \cap \{ \nabla f \neq 0  \}$.

\smallskip
Then for each point $p_0$ in $U$, there exists a neighborhood $V$ of $p_0$ in $U$ with coordinates $(s,t, x_3, x_4)$  such that $\nabla s= \frac{\nabla f }{ |\nabla f |}$ and $g$ can be written on $V$ as

\begin{equation} \label{mtr1a}
g= ds^2 +  p(s)^2  dt^2 +    h(s)^2 \tilde{g},
\end{equation}
 where  $p:=p(s)$ and $h:=h(s)$ are smooth functions and
 $\tilde{g}$ is (a pull-back of) a Riemannian metric on a $2$-dimensional domain with $x_3, x_4$ coordinates.

 We  get   $E_1 =\frac{\partial }{\partial s} $ and  $E_2 =\frac{1}{p} \frac{\partial }{\partial t} $.

\end{lemma}

 \begin{proof} Let $D^1$ be the 2-dimensional distribution spanned by $E_1 = \nabla s$ and $E_2$. Also let $D^2$ be the one spanned by $E_3$ and $E_4$.   Then $D^1$ and $D^2$ are both integrable by Lemma \ref{claim112na}.
We may consider the coordinates  $(x_1, x_2, x_3, x_4)$ from Lemma \ref{claim112}, so that $D^1$ is tangent to the 2-dimensional level sets

 $ \{ (x_1, x_2, x_3,x_4) |  \ x_3, x_4 \ {\rm constants} \} $ and $D^2$ is tangent to the level sets $ \{ (x_1, x_2, x_3,x_4) |  \ x_1, x_2 \ {\rm constants} \} $.
 As $D^1$ and $D^2$ are $g$-orthogonal, we can get the metric description for $g$ as follows;

$g= g_{11}dx_1^2 +   g_{12} dx_1 \odot dx_2  +  g_{22}dx_2^2 +    g_{33}dx_3^2 +   g_{34} dx_3 \odot dx_4 + g_{44}dx_4^2 $, where $\odot$ is the symmetric tensor product and $g_{ij}$ are functions  of $(x_1, x_2, x_3, x_4)$.

\medskip
 As $E_1= \nabla s   \in D^1$, we have $ds= g(E_1, \cdot)$.
  We define a $1$-form $\omega_2 ( \cdot ):= g ( E_2,  \cdot)$.
One can readily see that
  $ds^2 +   \omega_2^2  = g_{11}dx_1^2 +   g_{12} dx_1 \odot dx_2  +  g_{22}dx_2^2$.  In fact, one may feed $(E_i, E_j)$ to both sides and use the fact that each of $E_1$ and $E_2$ is of the form $a\partial_1 + b\partial_2$ as
 they are tangent to the sets $ \{ (x_1, x_2, x_3,x_4) |  \ x_3, x_4 \ {\rm constants} \} $,
   while each of $E_3$ and $E_4$ is of the form $c\partial_3 + d\partial_4$ for a similar reason; here we have set  $ \partial_{i} :=  \frac{\partial}{\partial x_i}$.

Recalling $[E_1, E_2] =- \zeta_2 (s)E_2$, we define a function $p(s) = e^{\int_{s_0}^s \zeta_2(u) du}$ for a constant $s_0$ so that  $\zeta_2=   \frac{p^{'}}{p}$.
Then, the 2-form $d (\frac{\omega_2}{p}) $ satisfies
$d (\frac{\omega_2}{p})(E_1, E_2) = -\frac{dp \wedge \omega_2 }{p^2}(E_1, E_2)  + \frac{1}{p} d \omega_2(E_1, E_2) =- \frac{p^{'}}{p^2}+ \frac{p^{'}}{p^2} = 0$.
And for $i \in \{3,4 \}$ and  for any $j  \in \{1,2,3,4 \}$,
 $d (\frac{\omega_2}{p})(E_i, E_j) = -\frac{dp \wedge \omega_2 }{p^2}(E_i,  E_j)  + \frac{1}{p} d \omega_2(E_i,  E_j) =  \frac{1}{p} d\omega_2(E_i,  E_j)= -\frac{1}{p} \omega_2([E_i,  E_j]) =  0$ by Lemma \ref{claim112na}.

 So, $d (\frac{\omega_2}{p})=0$ and   $\frac{\omega_2}{p} = d t$ for some function $t$ modulo a constant in a neighborhood of $p_0$.
 The metric $g$ can be now written as
\begin{equation} \label{ggg}
g= ds^2 +   p(s)^2 dt^2 +    g_{33}dx_3^2 +   g_{34} dx_3 \odot dx_4 + g_{44}dx_4^2,
\end{equation}
where $g_{ij}$ are functions  of $(x_1, x_2, x_3, x_4)$.
In the coordinates system $(s,t, x_3, x_4)$,  one easily gets $E_1 =\frac{\partial }{\partial s} $ and $E_2 =\frac{1}{p} \frac{\partial }{\partial t} $.

\medskip
 Now we use new coordinates $(s,t, x_3, x_4)$  in computations below, so that $\partial_{1}=\frac{\partial }{\partial s}$ and  $\partial_{2} =\frac{\partial }{\partial t}$, etc..
 From Lemma \ref{claim112na},
we have $ <\nabla_{E_i}  E_{j} ,  E_2>=0  $ for $i, j \in \{ 3,4 \}$.
As $ \partial_{3}$ and $\partial_{4}$ are both of the form $\gamma E_3 + \delta E_4$,
 we have that $< \nabla_{\partial_{i}}  \partial_{j} ,  \partial_2>= 0 $ for $i,j \in \{ 3,4  \}$.

We set  $ g_{ij}= g( \partial_i , \partial_{j}  )$. Due to (\ref{ggg}), for $i,j \in \{ 3,4  \}$;
 \begin{eqnarray} \label{what03}
0 & =   < \nabla_{\partial_{i}}  \partial_{j} ,  \partial_{2}> =   \sum_{k=1}^4 <\Gamma^{k}_{ij}  \partial_k ,  \partial_{2}  > \nonumber \hspace{2cm} \\
 &=   \sum_{k, l=1}^4 <  \frac{1}{2} g^{kl}( \partial_i g_{l{j}}+ \partial_j g_{l{i}} - \partial_l g_{ij} )\partial_k ,  \partial_{2}  > \hspace{0.7cm}   \nonumber   \\ & =- \sum_{k, l=1}^4  \frac{1}{2} g^{kl}  \partial_l g_{ij} <\partial_k ,  \partial_{2}  >   =  -\frac{1}{2}  \partial_{2} g_{ij}. \hspace{1.3cm}
\end{eqnarray}
We have shown;
\begin{eqnarray} \label{bbaa}
\frac{\partial g_{33}}{\partial t} =  \frac{\partial g_{34}}{\partial t} =\frac{\partial g_{44}}{\partial t}=0.
\end{eqnarray}
We consider the second fundamental form of a leaf for $D^2$ with respect to $E_1$;
$H^{E_1}  ( u , u ) =  -  < \nabla_{u} u ,  E_1>  $. As $D^2$ is totally umbilic by Lemma  \ref{derdlem} {\rm (ii)}, $H^{E_1} ( u , u ) =    \zeta \cdot g( u , u) $ for some function  $\zeta$ and any $u$ tangent to $D^2$.
Then, $H^{E_1} (E_3, E_3 ) = -  < \nabla_{E_3} E_3 ,  E_1> =    \zeta_3   $
So, $\zeta= \zeta_3 $, which is a function of $s$ only by Lemma \ref{raas}.

For $i, j \in \{ 3,4 \}$,  we compute similarly as in (\ref{what03}),
\begin{eqnarray*}
\zeta_3  g_{ij}& =H^{E_1} ( \partial_i , \partial_{j} )  =   -  < \nabla_{\partial_i}  \partial_{j} ,   \frac{\partial }{\partial s}> =  -  <\sum_k \Gamma^{k}_{i{j}}  \partial_k ,   \frac{\partial }{\partial s}  > \\
& =  - \sum_k <  \frac{1}{2} g^{kl}( \partial_i g_{lj} +\partial_{j} g_{li} - \partial_l g_{ij} )\partial_k ,  \frac{\partial }{\partial s} >       = \frac{1}{2} \frac{\partial }{\partial s} g_{i{j}}.
\end{eqnarray*}
 So,  $\frac{1}{2} \frac{\partial }{\partial s} g_{i{j}} =   \zeta_3  g_{ij}$. Integrating it, for $i, j \in \{ 3,4 \}$, we get $ g_{ij} = e^{C_{ij}} h(s)^2$. Here the function $h(s)>0$ is independent of $i,j$ and each function $C_{ij}$ depends only on $x_3, x_4$ by (\ref{bbaa}).

 Now  $g$ can be  written as $g= ds^2 +   p(s)^2 dt^2 +    h(s)^2 \tilde{g} $, where $\tilde{g}$ can be viewed as a Rimannian metric in a domain of $(x_3, x_4)$-plane.
 \end{proof}

\section{Analysis of the metric when $\lambda_2 \neq \lambda_3 =\lambda_4$}

We shall study more about the metric $g = ds^2 + p(s)^2  dt^2 + h(s)^2  \tilde{g}$ of (\ref{mtr1a}) obtained in Lemma \ref{claim112b}.

\begin{lemma} \label{112typee}
Let $(M, g, f)$ be a four dimensional gradient Ricci soliton with harmonic Weyl curvature which satisfies
the hypothesis of Lemma \ref{claim112b}.
For the metric $g = ds^2 + p(s)^2  dt^2 + h(s)^2  \tilde{g}$ of (\ref{mtr1a}), the two dimensional metric $\tilde{g}$ has constant curvature, say $k$.

\end{lemma}

\begin{proof}
In local coordinates $(x_1:=s, \ x_2:=t, \  x_3, x_4)$ of  Lemma \ref{claim112b},  we write some Christofel symbols $\Gamma_{ij}^k$ and Ricci curvature of $g$.  In this proof, for any $(0,2)$-tensor $P$,  $P( \frac{\partial}{\partial_{x_i}}, \frac{\partial}{\partial_{x_j}} )$ shall be denoted by $P_{ij}$.  We let $\tilde{\nabla}$,   $\tilde{\Gamma}_{ij}^k$ and $R^{\tilde{g}}_{ij}$ be the Levi-Civita connection,  Christofel symbols and Ricci curvature of $\tilde{g}$, respectively.
 For $i,j,k \in \{ 3,4 \}$, we get;
\begin{eqnarray}
  \Gamma_{ij}^k  =  \tilde{\Gamma}_{ij}^k  \nonumber   \hspace{4.8cm} \\
R_{ij} =  - \tilde{g}_{ij} \{ h h^{''}   +   \frac{p^{'}}{p} h h^{'}    + { h^{'}}^2 \}  +R^{\tilde{g}}_{ij}.  \label{riccicompob}
 \end{eqnarray}
From (\ref{riccicompob}),  for $i,j,k \in \{ 3,4 \}$,  we have $\nabla_k  \tilde{g}_{ij}=\tilde{\nabla}_k  \tilde{g}_{ij}=0$
and  $\nabla_k  R^{\tilde{g}}_{ij}=\tilde{\nabla}_k  R^{\tilde{g}}_{ij}$
so that $\nabla_k R_{ij} = \tilde{\nabla}_k R^{\tilde{g}}_{ij}$.
The condition $\delta W=0$ gives
$\nabla_k R_{ij}  -\nabla_j R_{ik} = -  \frac{R_j}{6} g_{ki} + \frac{R_k}{6} g_{ij}$.
For $i, j, k  \in \{ 3,4 \}$,  $R_j=R_k=0$, so  $\nabla_k R_{ij} =\nabla_j R_{ik}$.

Then, we get $\tilde{\nabla}_k R^{\tilde{g}}_{ij} = \tilde{\nabla}_j R^{\tilde{g}}_{ik} $. By the contracted second Bianchi identity the 2-dimensional metric $\tilde{g}$ then has constant curvature.

\end{proof}

The metric $ \tilde{g}$ of Lemma \ref{112typee} is locally isometric to
 the Riemannian metric $g_0= dr^2 + u(r)^2 d \theta^2$ on a domain in $\mathbb{R}^2$ with polar coordinates $(r, \theta)$, where $u(r)=r$ when $k=0$,  $u(r) =\sin(\sqrt{k} \cdot r) $  when $k>0$ or  $ u(r)=\sinh(\sqrt{-k} \cdot r)$ when $k<0$.
We may identify  $ \tilde{g}$ with  $g_0$ locally and
set $e_3 = \frac{\partial}{\partial r}$ and $e_4 = \frac{1}{u(r)} \frac{\partial}{ \partial \theta }$, which then form an orthonormal basis of $\tilde{g}$.

\begin{lemma} \label{112typeb1}
For the local soliton metric $g = ds^2 + p(s)^2  dt^2 + h(s)^2  \tilde{g}$ of (\ref{mtr1a}) obtained in Lemma \ref{claim112b} with the metric $ \tilde{g}$ of constant curvature $k$,  if we set $E_1 = \frac{\partial }{\partial  s}$, $E_2 = \frac{1}{p(s)}\frac{\partial }{ \partial  t}$,  $E_3 =  \frac{1}{h(s)} e_3$ and $E_4 =  \frac{1}{h(s)} e_4$, where $e_3$ and $e_4$ are  as in the above paragraph, then the connection form, Ricci and scalar curvature of $g$ are as below. Here $R_{ij} = R(E_i, E_j)$ and $R_{ijkl} = R(E_i, E_j, E_k, E_l)$.

\bigskip
$\nabla_{E_1}E_i = 0$, for $i=1,2,3,4$.

$\nabla_{E_i}E_1 =   \zeta_iE_i$, for $i=2,3,4$ with $\zeta_2= \frac{p^{'}}{p}$, $\zeta_3=\zeta_4=  \frac{h^{'}}{h}$.

$\nabla_{E_2}  E_2 = -\zeta_2  E_1$, $\ \ \ \ \nabla_{E_3}  E_3 = -\zeta_3  E_1  $, $\ \ \ \  \nabla_{E_4}  E_4 = -\zeta_4  E_1 + \beta_4 E_3  $.

 $\nabla_{E_2}E_3 =\nabla_{E_3}E_2 = \nabla_{E_4}E_2  =   \nabla_{E_2}E_4 =   0$.

$\nabla_{E_3}E_4 = 0$, $\ \ \ \ \nabla_{E_4}E_3 =  -\beta_4 E_4,  \ $ where $  \ \beta_4 = \frac{u^{'}(r)}{h(s)u(r)}$.
\begin{eqnarray}  \label{ricci34}
R_{1221} &= -  \frac{p^{''}}{p}=-\zeta_2^{'}
 -  \zeta_2^2 \ = \ R_{1ii1} =-\zeta_i^{'}
 -  \zeta_i^2= -  \frac{h^{''}}{h},   \ \ {\rm  for}  \ i \geq 3. \hspace{1.8cm}   \nonumber \\
R_{11} & =  -3\zeta_2^{'}
 -  3\zeta_2^2  = - 3 \frac{h^{''}}{h}. \hspace{7.7cm}  \nonumber \\
R_{22} & =  -\zeta_2^{'}
 -  \zeta_2^2 -2\zeta_2 \zeta_3 = -\frac{h^{''}}{h} -2 \frac{p^{'}}{p} \frac{h^{'}}{h} . \hspace{5.5cm}  \nonumber  \\
R_{33} &=R_{44} =     -\zeta_3^{'}
 -  \zeta_3^2 -\zeta_3 \zeta_2  -(\zeta_3)^2  + \frac{k}{h^2}= -\frac{h^{''}}{h} - \frac{p^{'}}{p} \frac{h^{'}}{h} - \frac{(h^{'})^2}{h^2}  + \frac{k}{h^2}. \hspace{0.5cm}  \nonumber \\
R_{ij}  &=0, \ \ \ \ \ {\rm  if} \ \   i \neq j.  \hspace{8.4cm} \nonumber \\
R \ &= -6\zeta_3^{'}
 -  6\zeta_3^2 -4\zeta_3 \zeta_2  -2(\zeta_3)^2  + 2\frac{k}{h^2}= - 6 \frac{h^{''}}{h} -4  \frac{p^{'}}{p} \frac{h^{'}}{h}  - 2\frac{(h^{'})^2}{h^2}  + 2\frac{k}{h^2}     \nonumber
\end{eqnarray}

\end{lemma}

\begin{proof} One may verify all the formulas by direct computation. In particular $\zeta_2= \frac{p^{'}}{p}$ and $\zeta_3=\zeta_4=  \frac{h^{'}}{h}$.  We get $\frac{p^{''}}{p} =  \frac{h^{''}}{h}$ from  (\ref{0110g}).

\end{proof}

What emerges from above discussions can be highlighted as
the following soliton on an open set, which results from Lemma \ref{claim112b} and \ref{112typee};

A four dimensional gradient Ricci soliton $(M, g, f)$ with harmonic Weyl curvature
has a connected coordinate neighborhood $(V, (s,t, x_3, x_4))  \subset M_{\mathcal{A}} \cap \{ \nabla f \neq 0  \}$, in which

\begin{equation} \label{mtr1}
g= ds^2 +  p(s)^2  dt^2 +    h(s)^2 \tilde{g}\ \ \ \ {\rm  on} \ V,
\end{equation}
where
 $\tilde{g}$ is a 2-dimensional Riemannian metric of constant curtvature $k$ on an $(x_3, x_4)$-domain.
 We have the adapted frame fields
  \begin{eqnarray} \label{mtr11}
\ \ \ \ \ \ \ \ \  E_1 = \frac{\nabla f }{ |\nabla f |}=\frac{\partial }{\partial s}, \ \  \ \ E_2 =\frac{1}{p} \frac{\partial }{\partial t},  \ \ \  E_3 =\frac{1}{h} e_3, \ \ \  E_4 =\frac{1}{h} e_4  \ \ {\rm  on} \ V, \nonumber \\
\ \      {\rm and} \ \ \lambda_2   \neq \lambda_3=  \lambda_4,   \hspace{7cm}
\end{eqnarray}
 where $e_3$ and $e_4$ are an orthonormal frame fields of $ \tilde{g}$ as in Lemma \ref{112typeb1}.

\bigskip

\begin{remark} \label{realan}
As mentioned in Section 2, $g$ and $f$ are real analytic (in harmonic coordinates), so is $|\nabla f|$ where $\nabla f \neq 0$. The Ricci eigenvalues $\lambda_i$ are real analytic in $M_{\mathcal{A}} \cap \{ \nabla f \neq 0  \}$.
So are $\zeta_i(s)=\frac{ 1}{|\nabla f |} (\lambda -   \lambda_i ) $.

Also $R^{'}= dR(E_1)$ is real analytic since it equals $dR( \frac{\nabla  f}{|\nabla  f|} )$. From (\ref{0110g}) $R(E_1, E_2, E_2, E_1) $ is real analytic.
As $-\zeta_2^{'} - \zeta_2^2=-\zeta_3^{'} - \zeta_3^2  =R(E_1, E_2, E_2, E_1)$,   $\zeta_2^{'}$ as well as $\zeta_3^{'}$  are real analytic.

\medskip
To exploit the real analyticity, we shall use the following simple fact; if $P \cdot Q$ equals zero (identically)  on an open connected set $W$ for two real analytic functions $P$ and $Q$, then either $P$ equals zero on $W$ or $Q$ equals zero on $W$.
\end{remark}

For the rest of this section we denote $a:=\zeta_2$ and $b:=\zeta_3$ for convenience.

In the adapted frame field $\{ E_i \}$ of (\ref{mtr11}),
we can write components of the soliton equation $\nabla d f(E_i, E_i) =  -(Rc- \lambda g)(E_i, E_i)$, $i=1,2,3$ as follows;
\begin{eqnarray}
f^{''}  &=   3a^{'}
 + 3a^2 +  \lambda. \hspace{2.2cm} \label{soliton00} \\
f^{'} a &=      a^{'}
 +  a^2 +2b a   + \lambda .  \hspace{1.4cm}  \label{solitonii} \\
 f^{'} b &=      b^{'}
 + b^2 +b a  +b^2  -\frac{k}{h^2}   + \lambda.  \label{solitonjj}
 \end{eqnarray}

\bigskip
In the next section we are going to deduce several linear or quadratic equations in $a$ and $b$ from (\ref{soliton00})-(\ref{solitonjj})
and $\delta W=0$.
But before we get to it, in the next three lemmas we shall understand three linear cases (when $a=0$, $b=0$ and  $a+b=0$  on a domain).

\begin{lemma} \label{bb1}
For the soliton metric $g$ of (\ref{mtr1})   with harmonic Weyl curvature and with the adapted frame fields (\ref{mtr11}),
the function $a$  cannot vanish on $V$.
\end{lemma}

\begin{proof}
If $a=0$, then $b^{'} + b^2 = a^{'} + a^2=0$. Integrate for $b= \frac{h^{'}}{h}$ to get $\frac{h^{'}}{h}=\frac{1}{s-c}$ for a constant $c$,  as $b \neq a=0$.  So, $h = c_h (s-c)$, for a constant $c_h \neq 0$.  From (\ref{solitonii}), $\lambda=0$. From  (\ref{soliton00}), $f^{''} = 0$ and $f^{'}$ is constant. From (\ref{solitonjj}) we get $f^{'}= \frac{1}{s-c} (1 - \frac{k}{c_h^2} )$. Then, $c_h^2=k>0$ and $f^{'} =0$. So, $g$ is Einstein, a contradiction to the hypothesis $\lambda_2   \neq \lambda_3.$
\end{proof}

\begin{lemma} \label{bb1}
For the soliton metric $g$ of (\ref{mtr1})   with harmonic Weyl curvature and with the adapted frame fields (\ref{mtr11}), assume that $b=0$ on $V$. Then
 $g$ is locally
 isometric to a domain in $ \mathbb{R}^2 \times (N, \tilde{g})$ with $g= ds^2 + s^2 dt^2+  \tilde{g} $, where $\tilde{g}$ is a Riemannian metric of constant curvature $\lambda \neq 0$ on a two dimensional manifold $N$.
And $f = \frac{\lambda}{2}s^2 +C_1$, for a nonzero constant $C_1$.
\end{lemma}

\begin{proof}
If $b=0$, then $a^{'} + a^2=0$. Integrate for $a= \frac{p^{'}}{p}$ to get $ \frac{p^{'}}{p}=\frac{1}{s-c_1}$ for a constant $c_1$,  as $a \neq b=0$.  So, $p = c_p (s-c_1)$, for a constant $c_p \neq 0$.   As $h$ is constant, we set $h=h_0>0$.

From  (\ref{solitonii}), $f^{'} =  \lambda (s-c_1)$.
We get $f(s) = \frac{1}{2} \lambda (s-c_1)^2+C_1$.  If $\lambda=0$, then $f$ is constant and $g$ is Einstein, which violates  $\lambda_2   \neq \lambda_3$ hypothesis.  So, $\lambda \neq 0$.
 From (\ref{solitonjj}), we have $ \frac{k}{h_0^2}=\lambda$. And by absorbing a constant to the variable $t$, we can write the metric $g= ds^2 + (s-c_1)^2 dt^2+  h_0^2 \tilde{g} $, where $h_0^2\tilde{g}$ is a Riemannian metric of constant curvature $ \frac{k}{h_0^2}=\lambda$.  The metric $g$ is isometric to $ ds^2 + s^2 dt^2+   h_0^2  \tilde{g} $. This proves the lemma.
\end{proof}

\begin{lemma} \label{bb6}
For the soliton metric $g$ of (\ref{mtr1})   with harmonic Weyl curvature and with the adapted frame fields (\ref{mtr11}), the function $a+b$  cannot vanish on $V$.
\end{lemma}

\begin{proof} Suppose $a+b=0$ on $V$. Then $a^{'} -b^{'} = b^2 -a^2 =0$.  So, $a-b = C$, a constant.
   Then $a=  \frac{p^{'}}{p}= \frac{C}{2}, b= \frac{h^{'}}{h}=  -\frac{C}{2}$. As $a \neq b$, $C \neq 0$.
  Then $h=  c_h e^{ -\frac{C}{2}s}$ for a constant $c_h>0$.
  Put it into (\ref{solitonii}) and  (\ref{solitonjj}), and we have
  $ k  = \lambda=0$ and $f^{'}$ is a constant. Then (\ref{soliton00}) gives
$C^2=0$, which is a contradiction.
\end{proof}

\section{Characterization of the metric when $\lambda_2 \neq \lambda_3 =\lambda_4$}

In this section we shall characterize the soliton metric $g$ of (\ref{mtr1})   with harmonic Weyl curvature and with the adapted frame fields (\ref{mtr11}).

\smallskip
From (\ref{solitonii}) and (\ref{solitonjj}),
\begin{equation} \label{2356}
 (a -b )  f^{'} =   b (a - b) + \frac{k}{h^2}.
 \end{equation}

\noindent Differentiating,
$ (a -b )^{'}   f^{'} + (a -b )   f^{''}=   b^{'} (a - b)  +b (a - b)^{'} -2\frac{k h^{'}}{h^3}.    $

\noindent Meanwhile, from (\ref{soliton00}), (\ref{2356}) and  $a^{'} -b^{'}=-a^{2} +b^{2}  $,
\begin{eqnarray*}
(a -b )^{'}   f^{'} + (a -b )   f^{''} = &  -(a^{2} -b^{2} )   f^{'} + (a -b )   (  - \lambda_1 + \lambda)  \hspace{3cm} \\
=&
  (a +b )  \{  -b (a - b) - \frac{k}{h^2}   \} + (a -b )   (  - \lambda_1 + \lambda). \hspace{0.6cm}
\end{eqnarray*}
So, we get

$   b^{'} (a - b)  +b (b^2 - a^2) -2\frac{k h^{'}}{h^3}= (a +b )  \{  -b (a - b) - \frac{k}{h^2}   \} + (a -b )   (  - \lambda_1 + \lambda).$
Then, as $b = \frac{h^{'}}{h}$,
 \begin{eqnarray*}
b^{'} (a - b) = & (a +b )  \{  - \frac{k}{h^2}   \} + 2\frac{k h^{'}}{h^3} + (a -b )   (  - \lambda_1 + \lambda) \\
  = &  (a -b )  \{  - \frac{k}{h^2}   \} + (a -b )   (  - \lambda_1 + \lambda). \hspace{1.5cm}
\end{eqnarray*}

\noindent As $\lambda_2 \neq \lambda_3$, we have $a -b  \neq 0$. We then have;

\begin{equation} \label{c120}
   2(b^{'} + b^2) + b^2   -  \frac{k}{h^2}
 + \lambda =0.
 \end{equation}

\noindent
From (\ref{solitonii}), (\ref{solitonjj}) and $b^{'} = a^{'} +a^2 - b^2 $, we have

$b(   a^{'}
 +  a^2 +2b a   + \lambda) = a(     b^{'}
 + b^2 +b a  +b^2  -\frac{k}{h^2}   + \lambda ),$
and so
\begin{equation}  \label{c102}
-(a-b)a^{'} -a^3  +ab^2  + \lambda(b-a)= -a\frac{k}{h^2}.
 \end{equation}

\medskip
 Next, we shall exploit the harmonic Weyl curvature condition.
In $\{ E_i \}$, we have
$\nabla_k R_{ij}  -\nabla_j R_{ik} = -  \frac{R_j}{6} g_{ki} + \frac{R_k}{6} g_{ij}$. Then as $\nabla_{E_1}E_2 =  \nabla_{E_1}E_3=0$,
\begin{eqnarray} \label{solitonii81}
0 & =  \nabla_1 R_{22}  - \nabla_2 R_{12}   - \frac{R^{'}}{6} \hspace{4.4cm} \nonumber \\
&=   \nabla_1  (R_{22})   +  R ( \nabla_{E_{2} }  E_{1},   E_{2})   + R ( \nabla_{E_2 }E_2,   E_1) - \frac{R^{'}}{6} \nonumber \\
& =   ( R_{22})^{'}  +
a R_{22} -a R_{11} - \frac{R^{'}}{6}. \hspace{3.3cm}
\end{eqnarray}
\begin{eqnarray} \label{solitonii83}
0 & =  \nabla_1 R_{33}  - \nabla_3 R_{13} - \frac{R^{'}}{6} \hspace{4.4cm} \nonumber \\
 &=   \nabla_1  (R_{33})   +  R ( \nabla_{E_3 }  E_1,   E_3)   + R ( \nabla_{E_3 }E_3,   E_1)  - \frac{R^{'}}{6}\nonumber \\
& =   (  R_{33})^{'}   +b R_{33} -b R_{11}  - \frac{R^{'}}{6}. \hspace{3.4cm}
\end{eqnarray}

\noindent Subtracting (\ref{solitonii83}) from (\ref{solitonii81}), with Lemma \ref{112typeb1}  we get

$( -a b     +b^2  - \frac{k}{h^2}  )^{'}  +
a(- a^{'}
 -  a^2 -2a b )  -(a -b)(-3a^{'}
 -  3a^2 )     -b ( - b^{'}
 -  b^2 -b a  -b^2  + \frac{k}{h^2} )  =0$,
 from which we obtain
 \begin{equation} \label{c101}
 -(a-b)a^{'} -a^3 + b^3 + 2a^2b -2ab^2 = b\frac{k}{h^2}.
 \end{equation}

 \medskip
\noindent Subtracting  (\ref{c102}) from (\ref{c101}),

\begin{equation}  \label{c112}
(a-b) ( 2ab  -b^2+\lambda )= (a+b)\frac{k}{h^2}.
 \end{equation}

\begin{lemma} \label{bb8}
For the soliton metric $g$ of (\ref{mtr1})  with harmonic Weyl curvature and with the adapted frame fields (\ref{mtr11}), assume that $k \neq 0$. Then the following holds;
\begin{equation} \label{c113}
b(\lambda + 3ab )(\lambda -2a^2 +ab  )=0.
\end{equation}
\end{lemma}

\begin{proof} We start from (\ref{c112}). Our hypothesis $k \neq 0$ and Lemma \ref{bb6} implies that  $ 2ab  -b^2+\lambda $ does not vanish.
So, we may take the natural log of (\ref{c112}) and differentiate it;

 \begin{equation}
-a-b + \frac{2a^{'}b+ 2ab^{'} - 2bb^{'}}{2ab - b^2 + \lambda}   = \frac{a^{'} + b^{'}}{a+b} -2b. \nonumber
\end{equation}

Then put $b^{'} = a^{'} +a^2 - b^2 $ into it;
  \begin{eqnarray*}
\frac{   a a^{'}+ (a - b )(a^2 -b^2)  }{2ab - b^2 + \lambda}= \frac{a^{'}+a^2 -b^2}{a+b}.
\end{eqnarray*}

Arranging terms, we obtain;

\begin{equation} \label{c119}
-a^{'} {(a^2+ b^2 -ab -\lambda) }  =   ( a^2-b^2)(a^2 -2ab -\lambda  ).
\end{equation}

\noindent Meanwhile, using that $a-b \neq 0$, from $b \times (\ref{c102}) + a \times (\ref{c101})=0$ we have

\begin{equation} \label{c104}
-(a+b) a^{'}  =  a^3   +2ab^2  +\lambda b.
\end{equation}

\noindent Removing $a^{'}$ in (\ref{c119}) and (\ref{c104}) and simplifying, we can get;
\begin{equation}
b(\lambda + 3ab )(\lambda -2a^2 +ab  )=0. \nonumber
\end{equation}
\end{proof}

We need to characterize the two equalities appearing in (\ref{c113}):   $\lambda + 3ab=0  $ and $\lambda -2a^2 +ab=0$.

\begin{lemma} \label{bb1a}
For the soliton metric $g$ of (\ref{mtr1})  with harmonic Weyl curvature and with the adapted frame fields (\ref{mtr11}),  assume that $k \neq 0$ and that  $h$ is not constant. Then $\lambda + 3ab$ does not vanish on $V$.
\end{lemma}
\begin{proof}
If $\lambda + 3ab=0$ vanishes, we have

 $0=a^{'}b + ab^{'}=(b^{'} +b^2 - a^2)b+ ab^{'}  =   (a+b)  (b^{'}+b^2 -ab ).$
By Lemma \ref{bb6},
we have $b^{'}+b^2 =ab =-\frac{\lambda}{3} .$
Due to (\ref{c120}),
$b^2 - \frac{k}{h^2} =-\frac{\lambda}{3}  $.
From (\ref{solitonjj}),
$f^{'} b = - \frac{\lambda}{3} - \frac{\lambda}{3}  - \frac{\lambda}{3} +   \lambda=0. $ As $h$ is not constant, we have $f^{'}=0$, a contradiction.
\end{proof}

We study the equation $\lambda -2a^2 +ab =0$;

\begin{lemma} \label{bb1b}
For the soliton metric $g$ of (\ref{mtr1})   with harmonic Weyl curvature and with the adapted frame fields (\ref{mtr11}), assume that $k \neq 0$ and that  $h$ is not constant.  Then $\lambda -2a^2 +ab$  does not vanish on $V$.

\end{lemma}

\begin{proof}
If $\lambda -2a^2 +ab =0$ on $V$,  put $\lambda =2a^2 -ab$ into (\ref{c104}) to get;

 $-a^{'}  =  \frac{a^3  +2a^2b +ab^2   }{  a+b}= a(a+b)$. So,  $a^{'} +a^2+  ab=0 $, i.e. $ p^{''}h + p^{'}h^{'} =0$.
 Integrating this, we get $p^{'} h =c_1$ for a constant $c_1$.
  As $ \frac{h^{''}}{h} = \frac{p^{''}}{p}$, we have  $ h^{''}p+p^{'}h^{'} =0$, which integrates to
 $h^{'} p =c_2$  for a constant $c_2$. As $a$ does not vanish by Lemma \ref{bb1} and $b \neq 0$ from hypothesis, $c_1c_2$ is not zero.
So $ \frac{h^{'}}{h} = \frac{c_2}{c_1} \frac{p^{'}}{p} $, i.e. $b= ca$, for $c \neq 0$. So,  $0=\lambda -2a^2 +ab = \lambda + (c-2)a^2 $.

\smallskip
If $ c \neq 2$, then $a$ is a nonzero constant.  $a^{'} +a^2+  ab=0 $ yields $a+b = 0$, which is not possible by Lemma \ref{bb6}.

\smallskip
If $c=2$, then $\lambda=0$ and
$ 2a =  b  $. Put these and $a^{'} +a^2+  ab=0 $ into (\ref{solitonii}) to get $f^{'} =2a$.
Then from (\ref{solitonjj}),  we get $k=0$, a contradiction.

\end{proof}

 \begin{lemma} \label{bb9c}
 For the soliton metric $g$ of (\ref{mtr1})  with harmonic Weyl curvature and with the adapted frame fields (\ref{mtr11}),
  assume that $k=0$.

Then $g$ is locally isometric to the metric $ds^2 + s^{\frac{2}{3}} dt^2+  s^{\frac{4}{3}} \tilde{g}$ on a domain of $\mathbb{R}^4$,
  where $ \tilde{g}$ is flat. Also,  $\lambda=0$ and $f=\frac{2}{3} \ln s + C_2$, for a constant $C_2$.

Furthermore, the Ricci curvature components and scalar curvature of $g$ are as follows;
 $R_{11} = \frac{2}{3s^2}$,  $R_{22} = -\frac{2}{9s^2}$, $R_{33}=R_{44} = -\frac{4}{9s^2}$, $R_{ij} =0$, $i \neq j$, and $R= -\frac{4}{9s^2} $. And the Weyl curvature of $g$ is not zero.

 \end{lemma}

 \begin{proof}  As $k=0$ and $a \neq b$,   $2ab -b^2 + \lambda=0$ from (\ref{c112}).
From the computation in Lemma \ref{112typeb1}, we get $ \  R =  -6 ( a^{'} + a^2)  -   8ab-2  \lambda $.
(\ref{solitonii81}) becomes;
\begin{eqnarray*}
0 & = -\{ a^{'} +a^2   +  2 ab  \}^{'}  -
a \{a^{'} + a^2  +   2ab  \} +3a ( a^{'} + a^2) \hspace{3.3cm}   \nonumber \\
&  - \frac{ 1}{6} \{  -6 ( a^{'} + a^2)  -   8ab-2  \lambda \}^{'}  \hspace{6.3cm} \nonumber \\
&= -\frac{2}{3} (ab)^{'} + 2a( a^{'}+ a^2-a b) \hspace{7.2cm} \nonumber \\
&= -\frac{2}{3} \{ a^{'}b + a(a^{'} + a^2 -b^2) \} + 2a( a^{'}+ a^2-a b)  \hspace{4.2cm} \nonumber \\
&= -\frac{2}{3} a^{'}b + \frac{4}{3} aa^{'} +  \frac{4}{3}  a^3  +  \frac{2}{3} ab^2 -2 a^2 b.  \hspace{5.9cm}
\end{eqnarray*}

We get;
\begin{eqnarray*}
(2a-b) (a^{'} + a^2 -ab)=0.
\end{eqnarray*}
If $a^{'}+ a^2 -ab=0$, we get  $p^{''} =  \frac{p^{'} h^{'}}{h}$. Then $\frac{p^{'}}{h}=c_1$, a constant.    From $ \frac{ h^{''}}{h}=\frac{p^{''}}{p}=  \frac{p^{'} h^{'}}{ph}$, we also get $\frac{h^{'}}{p}=c_2 $, a constant.  So, $ab =\frac{p^{'}h^{'}}{ph} =c_1c_2  $.  And $2ab - b^2 + \lambda=0$ tells that $b$ is a constant. If $b=0$, then $\lambda=k=0$ and from (\ref{solitonii}) $f^{'}a=0$. So, $f^{'}=0$ and $g$ is Einstein, a contradiction to the hypothesis.
Now $b$ is a nonzero constant. Then  $b^{'} + b^2 =a^{'}+ a^2=ab$ gives  $a=b$, a contradiction to the hypothesis.

\medskip

If $2a=b$, then  $0=   2ab - b^2 + \lambda = \lambda $.
From $a^{'} +a^2 = b^{'} +b^2 = 2a^{'} + 4a^2$, we get  $a^{'} +3a^2 =0$. Integrating it to get $a=\frac{p^{'}}{p}= \frac{1}{3s -c_2}$ for a constant $c_2$.
 (\ref{solitonii}) gives
$f^{'} a = 2a^2 $, so that $f^{'} = 2a=\frac{2}{3s -c_2} $.
As $2 \frac{p^{'}}{p} = \frac{ h^{'}}{h} $, we have $p^2 = e^c h$ for a constant $c$.
We get $p = e^{c_3} (3s- c_2)^{\frac{1}{3}}$ and $h= e^{c_4} (3s- c_2)^{\frac{2}{3}}$.

So,  $g$ is locally isometric to the metric
 $ds^2 + s^{\frac{2}{3}} dt^2+  s^{\frac{4}{3}} \tilde{g}$ on a domain of $\mathbb{R}^4$, where $ \tilde{g}$ is flat. And $f=\frac{2}{3} \ln s + C_2$, for a constant $C_2$.

  \medskip
  One can check that the above $(g,f)$ satisfy the soliton equation including (\ref{soliton00}), (\ref{solitonii}), (\ref{solitonjj}) and the harmonicity of Weyl curvature, and so is a steady Ricci soliton.
One can easily compute the curvature components of $g$.

 \end{proof}

Based on the real analyticity of $a,b,$ $ a^{'}$ and $b^{'}$ from Remark \ref{realan}, we combine the previous lemmas to obtain the next proposition.

\begin{prop} \label{prop2}
Let $(M, g, f)$ be a four dimensional gradient Ricci soliton with harmonic Weyl curvature.
Suppose that  $ \lambda_2  \neq \lambda_3=  \lambda_4$ for an adapted frame fields $E_j$, $j=1,2,3,4$,  in an open subset $U$ of $M_{\mathcal{A}} \cap \{ \nabla f \neq 0  \}$.

\smallskip
Then for each point $p_0$ in $U$, there exists a neighborhood $V$ of $p_0$ in $U$ with coordinates $(s,t, x_3, x_4)$  in which
$(V, g,f)$ can be one of the following;

\medskip
{\rm (i)} $(V, g)$ is isometric to a domain in $ \mathbb{R}^2 \times N$ with $g= ds^2 + s^2 dt^2+  \tilde{g} $, where $(N, \tilde{g})$ is a Riemannian manifold of constant curvature $\lambda \neq 0$.
 And $f = \frac{\lambda}{2} s^2+C_1$, for a constant $C_1$.

{\rm (ii)} $(V, g)$ is  isometric to a domain in $\mathbb{R}^4$ with the Riemannian metric
 $ds^2 + s^{\frac{2}{3}} dt^2+ s^{\frac{4}{3}} \tilde{g}$, where $ \tilde{g}$ is flat. Also,  $\lambda=0$ and $f=\frac{2}{3} \ln s+C_2$, for a constant $C_2$. The metric $g$ is not locally conformally flat.
\end{prop}

\begin{proof} We exploit the real analyticity. Lemma \ref{bb9c} settles the $k=0$ case.
Lemma \ref{bb8} divides the $k \neq 0$ case into three subcases $b=0$, $\lambda + 3ab=0$ and $\lambda -2a^2 +ab=0$ which are
treated in  Lemmas \ref{bb1}, \ref{bb1a} and  \ref{bb1b}, respectively.

\end{proof}

\section{4-dimensional soliton with $\lambda_2 = \lambda_3 =\lambda_4$.}

In this section we treat the remaining case of $\lambda_2 = \lambda_3 =\lambda_4$  for an adapted frame field.

 \begin{prop} \label{lcf1}
Suppose that $(M, g,f)$ is  a four dimensional gradient Ricci soliton with harmonic Weyl curvature and non constant $f$ and that $\lambda_2 = \lambda_3 =\lambda_4 \neq \lambda_1$ for an adapted frame field in an open subset $U$ of $M_{\mathcal{A}} \cap \{ \nabla f \neq 0  \}$.

\smallskip
Then for each point $p_0$ in $U$, there exists a neighborhood $V$ of $p_0$ in $U$ where  $g$ is a warped product;
\begin{equation} \label{metr}
g= ds^2 +    h(s)^2 \tilde{g},
\end{equation}
 for a positive function $h$,
  where the Riemannian metric $\tilde{g}$ has constant curvature, say $k$.  In particular, $g$ is locally conformally flat.
 \end{prop}

 \begin{proof}
Near $p_0$ in $U$, we use a local
 coordinates system $(x_1 := s, \ x_2, x_3, x_4)$ from Lemma \ref{threesolb} {\rm (v)} in which
 the metric
  $g = ds^2 + \sum_{i, j \geq 2}^{4}g_{ij} dx_i dx_j$ with $g_{ij} =g_{ij}(x_1, \cdots, x_4)$.

 By Lemma \ref{raas}, near $p_0$, each $\lambda_i$, $i = 1,2,3,4$ is a function of $s$ only.
We consider the second fundamental form of the level hypersurfaces $\Sigma_c$ of $f$ with respect to $E_1$;
$H^{E_1}  ( u , u ) =  -  \langle \nabla_{u} u ,  E_1\rangle  $. As $\Sigma_c$ is totally umbilic by Lemma \ref{derdlem} (ii), $H^{E_1} ( u , u ) =    G \cdot g( u , u) $ for any $u$ tangent to $\Sigma_c$ and some function  $G$.
Then, by Lemma \ref{derdlem} (i) $  \langle \nabla_{E_2} E_2 ,  E_1\rangle =    \frac{ \lambda_2^{'} -\frac{1}{6}R^{'} }{\lambda_2  - \lambda_1  }   $
So, $G= - \frac{ \lambda_2^{'} -\frac{1}{6}R^{'} }{\lambda_2  - \lambda_1  }$ is a function of $s$ only.

For $i, j \in \{2, 3,4 \}$,  setting $\partial_i:= \frac{\partial}{\partial x_i}$, we compute,
\begin{eqnarray*}
G(s) \cdot g_{ij}& =H^{E_1} ( \partial_i , \partial_{j} )  =   -  \langle \nabla_{\partial_i}  \partial_{j} ,   \frac{\partial }{\partial s}\rangle =  -  \langle\sum_{k=1}^4 \Gamma^{k}_{i{j}}  \partial_k ,   \frac{\partial }{\partial s}  \rangle \\
& =  - \sum_k \langle  \frac{1}{2} g^{kl}( \partial_i g_{lj} +\partial_{j} g_{li} - \partial_l g_{ij} )\partial_k ,  \frac{\partial }{\partial s} \rangle       = \frac{1}{2} \frac{\partial }{\partial s} g_{i{j}}.
\end{eqnarray*}
 So,  $\frac{1}{2} \frac{\partial }{\partial s} g_{i{j}} =   G(s)  g_{ij}$. Integrating it, we get $ g_{ij} = e^{C_{ij}} w(s)$. Here the function $w(s)$ is independent of $i,j$ and each $C_{ij}$ depends only on $x_2, x_3, x_4$.

 Now  $g$ can be  written as $g= ds^2 +    h(s)^2 \tilde{g} $, where $\tilde{g}$ can be viewed as a Riemannian metric in a domain of $(x_2, x_3, x_4)$-plane.

 \medskip
 To prove that  $\tilde{g}$ has constant curvature, we modify the proof of Derdzi\'{n}ski's Lemma 4  in \cite{De}, which is stated for harmonic curvature case.

 For $i,j \in \{2, 3,4 \}$, we compute the Christoffel symbols and Ricci curvature of $g$;
\begin{eqnarray}
\Gamma_{ij}^1 =  - h h^{'} \tilde{g}_{ij},  \ \ \ \  \ \ \ \  \Gamma_{1j}^i  =  \frac{h^{'}}{h} \delta_{ij}, \hspace{4.7cm} \nonumber  \\
 \ \ R_{1i}=0, \ \ \ \ \ \  R_{11}= -3 \frac{h^{''}}{h},  \ \ \ \ \
R_{ij} =  - \tilde{g}_{ij} ( h h^{''}   +   2 { h^{'}}^2 )  +R^{\tilde{g}}_{ij}.  \label{riccicompob3}
 \end{eqnarray}

The condition $\delta W=0$ gives
$\nabla_k R_{ij}  - \nabla_j R_{ik} = -  \frac{R_j}{6} g_{ki} + \frac{R_k}{6} g_{ij}$.
In particular, for $i, j  \in \{2, 3,4 \}$,  $\nabla_1 R_{ij}  -\nabla_j R_{i1} =  \frac{R_1}{6} g_{ij}$. From (\ref{riccicompob3}),
\begin{eqnarray*}
 \frac{\partial_1 R}{6} h^2 \tilde{g}_{ij}   =& \frac{\partial_1 R}{6} g_{ij}=\nabla_1 R_{ij}  -\nabla_{i} R_{1 j} \hspace{6cm} \\
=&   \partial_1  R_{ij}   -  R ( \nabla_{\partial_{1} }\partial_{j},   \partial_{i})   + R ( \nabla_{\partial_{i} }\partial_{j},   \partial_{1}) \hspace{4.2cm} \\
 =&  \partial_1  R_{ij}   -  \frac{ h^{'} }{h} R ( {\partial_{j}  },   \partial_{i}) - h h^{'} R (   \partial_{1 },   \partial_{1}) \tilde{g}( \partial_{i } , \partial_{j }  ) \hspace{3.3cm} \\
 =&  - \tilde{g}_{ij} \partial_1 (   h^{''} h  +  2 {h^{'}}^2 )  -  \frac{ h^{'} }{h} [  - \tilde{g}_{ij} ( h h^{''}     +2 { h^{'}}^2 )  +R^{\tilde{g }}_{ij} ] - h h^{'} R_{11} \tilde{g}_{ij}.
 \end{eqnarray*}

As $R$ depends only on $s$, so does $\partial_1 R= \frac{\partial R }{\partial s}$.
Therefore we get $R^{\tilde{g }}_{ij}=H(s) \cdot \tilde{g}_{ij}$ for a function $H(s)$ of $s$ only. So, $\tilde{g}$ is a 3-dimensional Einstein metric.
\end{proof}

 For the metric in (\ref{metr}), $h$ and  $f$ satisfy the following equations from $\nabla \nabla f + Rc = \lambda g$;
\begin{eqnarray}
f^{''} - 3\frac{h^{''}}{h}    =   \lambda,  \hspace{2.5cm} \label{fh} \\
\frac{h^{'}}{h} f^{'}  + \frac{2k}{h^2} -  \frac{h^{''}}{h}  - 2\frac{(h^{'})^2}{h^2} =   \lambda.\label{fh2}
\end{eqnarray}

\begin{remark} \label{rk1}
{\rm If all $\lambda_i$'s, $i=1, \cdots, 4$, are equal,  then the metric is Einstein. And if $f$ is not constant, then the conlusion of Proposition  \ref{lcf1} still holds.  In fact, from the section 1 of \cite{CC},
the Einstein metric $g$ becomes locally of the form $g= ds^2 +    (f^{'}(s))^2 \tilde{g}$ where $\tilde{g}$ has constant curvature. Then, the soliton can be seen to be either Gaussian  or a flat metric with  $\nabla d f=0$; see also Proposition 2 of \cite{PW}.}
\end{remark}

\section{Classification of gradient Ricci solitons with harmonic Weyl curvature}

 We are going to combine Proposition \ref{4dformc}, \ref{prop2} and \ref{lcf1} to  prove
Theorem \ref{local} after we settle the next lemma;
\begin{lemma} \label{notwo}
No two of the local four types of solitons {\rm (i)}$ \thicksim${\rm (iv)} in the statement of  Theorem \ref{local} can exist on a connected soliton.
\end{lemma}

\begin{proof}
When the real analytic function $f$ is constant in an open subset, then it is constant on $M$ as $M$ is connected.
So, if a soliton is the type {\rm (i)} in an open subset, it will be so on $M$.

If $g$ is a locally conformally flat metric on an open subset $U$ with non constant $f$, then   $|W|^2=0$  on $U$ and the real analytic function $|W|^2=0$ everywhere on $M$. So, $g$ is locally conformally flat on $M$ and $f$ is nowhere constant on $M$.  The types {\rm (ii)}  and  {\rm (iii)} do not satisfy  $|W|^2=0$.

If $g$ is isometric, on an open subset $V$, to a domain in $ \mathbb{R}^2 \times N_{\lambda}$,  then
$R=2 \lambda$ on $V$ and by real analyticity  $R=2 \lambda $  on $M$.
But if $g$ is isometric, on another open subset $W$, to the metric
 $ds^2 + s^{\frac{2}{3}} dt^2+ s^{\frac{4}{3}} \tilde{g}$, then the scalar curvature $ R= -\frac{4}{9s^2} $ is not locally constant.
 This proves the lemma.
 \end{proof}

\noindent {\bf Proof of Theorem \ref{local}  } Due to Lemma \ref{notwo} we may consider only one type on $M$.
When $f$ is constant, it corresponds to the type {\rm (i)}.

So, suppose that $f$ is not constant. Note that the statement {\rm (iv)} holds by Proposition \ref{lcf1} and Remark \ref{rk1}.
We denote the open dense subset $M_{\mathcal{A}} \cap \{ \nabla f \neq 0  \}$ by $K$.
 If $K=M$, then  the statements for  {\rm (ii)}  and {\rm (iii)} also hold from Proposition \ref{prop2}.

\smallskip
For the rest of proof we assume that there is a point $p_0 \in M \setminus K$.

 \medskip
When $(K, g)$ is of the type {\rm (ii)}, $(K, g)$ is locally isometric to $ \mathbb{R}^2 \times N_{\lambda}$, where the Ricci tensor is parallel.
 As $K$ is dense in $M$, the Ricci tensor is parallel near $p_0$ with eigenvalues $\lambda$ and $0$ of both multiplicity two by continuity. We can decompose the tangent bundle over a neighborhood of $p_0$; $TM = \eta_1 \oplus \eta_2$, where $\eta_1, \eta_2$ are $2$-dimensional parallel distributions with $Rc_{| \eta_1} = \lambda \cdot {\rm Id}$ and $Rc_{| \eta_2} = 0 \cdot {\rm Id}$.
 By de Rham decomposition theorem \cite[Section 8.3.1]{PP},
 $p_0$ has an open ball $B \subset M$ with $p_0$ as the center, where $B$ is isometric to (to be identified with) a ball in $ \mathbb{R}^2 \times N_{\lambda}$.
Now we can just solve for $f$ from the gradient soliton equation $ \nabla d f = -Rc + \lambda g$ to get;    $f= \frac{\lambda s^2}{2} + C$ where $s(\cdot) := d_{\mathbb{R}^2}( p_0,\cdot) $ is the Euclidean distance function from $p_0$. So, a neighborhood of $p_0$ is of type {\rm (ii)}.

\medskip
Suppose that $(K, g)$ is of the type {\rm (iii)}. Let  $\gamma_1: [0,1] \rightarrow M$ be a smooth path with $\gamma_1(0) =p_0 $ and $\gamma_1(1) \in  K$. Let $c \in [0,1)$ be the largest element in $\{ t  \in [0,1) \ | \  \gamma_1(t) \in M \setminus K \}$.  Define $\gamma$ to be the restriction of $\gamma_1$ on $[c,1]$. Set  $p:= \gamma(c)$ which is in $M \setminus K $.
Then $\gamma((c,1]) \subset  K$. Near any point $q \in \gamma((c,1])$, by Proposition \ref{prop2} we have local coordinates neighborhood $B_q \subset K$ with $(s_q, t, x_3, x_4)$ in which $f= \frac{2}{3} \ln (s_q) +C_q$ with the function $s_q$ and  constant $C_q$ depending on $q$. In a neighborhood $B_r \subset K$ of  another point $r \in \gamma((c,1]) $, we have a similar expression of $f= \frac{2}{3} \ln (s_r) +C_r$.
On a possible overlap region $B_q \cap B_r $, $\frac{2}{3} \ln (s_q) +C_q =  \frac{2}{3} \ln (s_r) +C_r$.
By taking its gradient, we have
 $ \frac{\nabla s_q}{s_q} =  \frac{\nabla s_r}{s_r}$. As $\nabla s_q= \frac{\nabla f}{|\nabla f|} =\nabla s_r$, we get
  $s_q =s_r$ and then $C_q= C_r$.

We may set $s:=s_q$ and  $C:=C_q$ which are independent of $q$ and  $f= \frac{2}{3} \ln (s) + C$ near $\gamma((c,1])$.
As $|\nabla s| \equiv 1$, the oscillation of $s$ along $\gamma$ is less than or equal to the length of $\gamma$, which is finite.
So, $|\nabla f| = \frac{2}{3s}$ cannot be zero at $p$. From Lemma \ref{bb1b}, the Ricci-eigen functions of $g$ are $\lambda_{1} = \frac{2}{3s^2}$,  $\lambda_{2} = -\frac{2}{9s^2}$, $\lambda_{3}=\lambda_{4} = -\frac{4}{9s^2}$. So, $p$ shall stay in $M_{\mathcal{A}}$ by definition. Then $p \in K$. This contradiction implies that $M \setminus K $ is an empty set.

Proposition \ref{4dformc} shows that there are no other types than  {\rm (i)}-{\rm (iv)}.
This proves the theorem.
 \hspace{8.7cm}  $\blacksquare$

\medskip
We remark that the incomplete steady gradient soliton in Theorem \ref{local} {\rm (iii)} has negative scalar curvature, in contrast to the fact that complete steady gradient solitons should have nonnegative scalar curvature.

\medskip

As a Corollary to Theorem \ref{local}, we state a classification of 4-dimensional complete gradient Ricci solitons with harmonic Weyl curvature. The case  Theorem \ref{local} {\rm (iii)} can only yield an incomplete soliton. And for case {\rm (ii)}, when $g$ is complete and  locally isometric to $ \mathbb{R}^2 \times N_{\lambda}$,  its universal cover is isometric to $ \mathbb{R}^2 \times N_{\lambda}$.

\begin{thm} \label{claim112na5}
Let $(M, g, f)$ be a complete four dimensional gradient Ricci soliton $\nabla df = -Rc + \lambda g$ with harmonic Weyl curvature.
Then it is one of the following;

\medskip
{\rm (i)} $g$ is an Einstein metric with $f$ a constant function.

{\rm (ii)} $g$ is isometric to a finite quotient of $ \mathbb{R}^2 \times N_{\lambda}$
where $ \mathbb{R}^2$ has the Euclidean metric and $N_{\lambda}$ is a 2-dimensional Riemannian manifold of constant curvature ${\lambda} \neq 0$.   And $f = \frac{\lambda}{2} |x|^2$ modulo a constant on the Euclidean factor.

{\rm (iii)} $g$ is locally conformally flat.

\end{thm}

Complete locally conformally flat {\it steady} gradient Ricci solitons are  classified to be either flat or isometric to the Bryant soliton, in \cite{CC1, CM}. This result and Theorem \ref{claim112na5} yield  Theorem \ref{steady}.
We also understand better complete expanding gradient Ricci solitons  with harmonic Weyl curvature as in  Theorem \ref{expand}.

\medskip
As mentioned in the introduction, we can show the local classification of gradient Ricci soliton with {\it harmonic curvature}  as a corollary of Theorem \ref{local}.

\begin{cor} \label{har2}
Let $(M, g, f)$ be a (not necessarily complete) four dimensional gradient Ricci soliton satisfying $\nabla df = -Rc + \lambda g$ with harmonic curvature. Then it is locally one of the three types {\rm (i)}-{\rm (iii)} below;  for each point $p$, there exists a neighborhood $V$ of $p$  such that
$(V, g,f)$ can be one of the following;

{\rm (i)} $g$ is an Einstein metric and $f$ is constant.

{\rm (ii)}  $g$ is isometric to a domain in $ \mathbb{R}^2 \times N_{\lambda}$
where $ \mathbb{R}^2$ has the Euclidean metric and $N_{\lambda}$ is a 2-dimensional Riemannian manifold of constant curvature ${\lambda} \neq 0$.   And $f = \frac{\lambda}{2} |x|^2$ modulo a constant on the Euclidean factor.

{\rm (iii)} $g$ is  isometric either to a domain in the Gaussian soliton or to a domain in $ \mathbb{R} \times M_{\lambda}$ with the product metric, where $M_{\lambda}$ is a 3-dimensional Riemannian manifold of constant curvature $ \frac{ \lambda}{2} \neq 0$, and $f = \frac{\lambda}{2} |x|^2$ modulo a constant on the Euclidean factor.

\end{cor}

\begin{proof} In this proof we do not rely on Theorem 1.2 of \cite{PW} as it works for a complete soliton.

The soliton metric $ds^2 + s^{\frac{2}{3}} dt^2+ s^{\frac{4}{3}} \tilde{g}$ in Theorem \ref{local} {\rm (iii)} does not have constant scalar curvature, so does not have harmonic curvature.

Note that the above
{\rm (iii)} should come from  Theorem \ref{local} {\rm (iv)}, in which the metric is of the form $ g=ds^2 +    h(s)^2 \tilde{g},$
where $\tilde{g}$ has constant curvature.
Lemma \ref{solitonformulas}  {\rm (ii)} gives $R + |\nabla f|^2 - 2\lambda f = {\rm constant}$. We differentiate with the local variable $s$ where $|\nabla f| \neq 0$, and get
$2 f^{'} f^{''} = 2\lambda f^{'}$ since $R$ is constant. So,   $f^{''} = \lambda$.  From (\ref{fh}), $h^{''} =0$.
Either $h=a$ or $h=bs$ for constants $a, b \neq 0$ after shifting $s$ by a constant.

When $h=a$, from (\ref{fh2}) we get $ \frac{k}{a^2} =  \frac{ \lambda}{2}$. We have $g = ds^2 +  \tilde{g}$ where $ \tilde{g}$
 has constant curvature $ \frac{ \lambda}{2}$. And we may set $f= \frac{\lambda}{2} s^2+C$ by shifting $s$.  As $f$ is not constant, $\lambda \neq 0$.

When $h=bs$, using (\ref{fh2}) and $f^{''} = \lambda $ we obtain that $f^{'} = \lambda s$ and $k= b^2$. We get $f = \frac{1}{2}\lambda s^2 +C$ so that
$\lambda \neq 0$.
And $g = ds^2 +  s^2 \tilde{g}, $ where $ \tilde{g}$ has constant curvature $+1$.  This yields the Gaussian soliton.

\medskip
(As an alternative to settle {\rm (iii)},  the section 2.2 of \cite{CM} may be cited. But that section is based on the existence of a self-similar solution, which exists if the soliton metric is complete \cite{Z2}. Here the metric may be incomplete.)
\end{proof}

\begin{remark}
{\rm In Theorem \ref{local} {\rm (iii)} we have got a four-dimensional incomplete soliton. One may ask if there exist {\it complete} non-conformally-flat gradient Ricci solitons of dimension$\geq 5$ with harmonic Weyl curvature and  $\lambda \leq 0$.

There are a number of objects to study by extending our method; it would be interesting to characterize the higher dimensional gradient Ricci solitons with harmonic Weyl curvature as well as other Ricci solitons.
Of course, other geometric structures than solitons can also be approached by the method here. }
\end{remark}

\begin{remark}
{\rm  There are a number of literatures on {\it orbifolds} in the theory of Ricci flow, for instance
\cite{CWL, KL}. As our result is a local description, it is possible to state an orbifold version of
Theorem \ref{claim112na5}.}
\end{remark}

\begin{remark}
{\rm B.L. Chen proved a local version of Hamilton-Ivey type estimate for three dimension in \cite{Ch}, which has been extended to $W=0$ case by Zhang \cite{Z}.
From Theorem \ref{local}, one may ask if such a local version still holds when $\delta W=0$.
}
\end{remark}


\begin{thebibliography}{99}

\bibitem{Be}  A.L. Besse,  {\em Einstein manifolds}. Ergebnisse der Mathematik, 3 Folge, Band 10, Springer-Verlag, 1987.

\bibitem{BM} J. Bernstein and T. Mettler, {\em Two-dimensional gradient Ricci solitons revisited}, International Mathematics Research Notices no. 1 (2015), 78-98.

\bibitem{Br} S. Brendle, {\em Rotational symmetry of self-similar solutions to the Ricci flow}, Invent. Math. 194(3) (2013), 731-764.

\bibitem{Cao1} H.D. Cao, {\em Recent progress on Ricci solitons}, Recent advances in geometric
analysis, 138, Adv. Lect. Math. (ALM), 11, Int. Press, Somerville, MA, 2010.

\bibitem{CCCMM} H.D. Cao, G. Catino, Q. Chen and C. Mantegazza,
and L. Mazzieri,  {\em  Bach-flat gradient steady ricci solitons}, Calc. Var. Partial Differential Equations 49 no. 1-2 (2014), 125-138.

\bibitem{CC1}  H.D. Cao and Q. Chen, {\em On locally conformally flat gradient steady Ricci solitons},
Trans. Amer. Math. Soc. 364 (2012), 2377-2391.

\bibitem{CC2} H.D. Cao and Q. Chen, {\em On Bach-flat gradient shrinking Ricci solitons}, Duke
Math. J. 162 (2013), 1003-1204.

\bibitem{CCZ} H.D. Cao, B.L. Chen and X.P. Zhu, {\em Recent developments on hamilton's ricci flow},   Surveys in differential geometry. Vol. XII. Geometric flows,  47-112, Surv. Differ. Geom., 12, Int. Press, Somerville, MA, 2008.


\bibitem{CWZ} X. Cao, B. Wang and Z. Zhang, {\em On locally conformally flat gradient shrinking Ricci solitons}, Commun. Contemp. Math.  13  no. 2  (2011), 269-282.


\bibitem{CM} G. Catino and C. Mantegazza, {\em The evolution of the Weyl tensor under the Ricci flow}, Ann. Inst. Fourier   (Grenoble)  61   no. 4 (2011), 1407-1435.


\bibitem{CC} J. Cheeger and T.H. Colding, {\em Lower Bounds on Ricci Curvature and the Almost Rigidity of Warped Products},
Annals of Mathematics, 2nd Ser., Vol. 144, No. 1 (Jul., 1996), 189-237.

\bibitem{Ch} B.L. Chen, {\em Strong uniqueness of the Ricci flow}, J. Differential Geom. 82 no. 2 (2009),
362-382.


\bibitem{CD} C.W. Chen and A. Deruelle, {\em Structure at infinity of expanding gradient Ricci soliton}, to appear on Asian J. Math.

\bibitem{CW} X.X. Chen, Y. Wang, {\em On four-dimensional anti-self-dual gradient Ricci solitons},
Jour. Geom. Anal.  25  no. 2  (2015), 1335-1343.


\bibitem{Cho} O. Chodosh, {\em Expanding Ricci solitons asymptotic to cones}, Calc. Var. Partial Differential Equations, 51(1-2) (2014), 1-15.


\bibitem{CWL} B.Chow and L.F. Wu, {\em The Ricci flow on compact 2-orbifolds with curvature negative somewhere}, Comm. Pure Appl. Math. 44 (1991), 275-286.

\bibitem{De} A. Derdzi\'{n}ski, {\em Classification of Certain Compact Riemannian Manifolds
with Harmonic Curvature and Non-parallel Ricci Tensor}, Math. Zeit. 172  (1980), 273-280.


\bibitem{ELM}  M. Eminenti, G. La Nave and C. Mantegazza, {\em Ricci solitons: the equation point of view}, Manuscripta Math., 127(3)  (2008), 345-367.

\bibitem{FG} M. Fern\'{a}ndez-L\'{o}pez and E. Garc\'{i}a-R\'{i}o, {\em Rigidity of shrinking Ricci solitons},
Math. Zeit. 269 (2011), 461-466.


\bibitem{HPW} C. He, P. Petersen and W. Wylie, {\em On the classification of warped product Einstein metrics},
Comm. in Anal. and Geom., v. 20, no. 2  (2012), 271-311.


\bibitem{Iv} T. Ivey, {\em Local existence of Ricci solitons}, Manuscripta Math. 91 (1996), 151-162.

\bibitem{Iv3} T. Ivey, {\em Ricci solitons on compact three-manifolds}, Differential Geom. Appl., 3(4) (1993), 301-307.


\bibitem{KL} B. Kleiner and J. Lott, {\em  Geometrization of three dimensional orbifolds via Ricci flow}, Asterisque  No. 365  (2014), 101-177.

\bibitem{MS} O. Munteanu and N. Sesum, {\em On gradient Ricci solitons}, Jour. Geom. Anal. 23
(2013), 539-561.

\bibitem{NW} L. Ni and N. Wallach, {\em On a classification of the gradient shrinking solitons}, Math. Res. Lett,
15  no. 5 (2008), 941-955.

\bibitem{P} G. Perelman, {\em The entropy formula for the Ricci flow and its geometric applications}, http://arxiv.org/pdf/math/0211159v1.pdf  (2002).

\bibitem{PP} P. Petersen,  {\em Riemannian Geometry},  Grad. texts in Math. 171, Springer-Verlag, 1998.


\bibitem{PW} P. Petersen and W. Wylie, {\em  Rigidity of gradient Ricci solitons}. Pac. J. Math. 241  (2009), 329-345.

\bibitem{PW2}  P. Petersen and W. Wylie, {\em On the classification of gradient Ricci solitons},  Geom. Topol. 14  no. 4 (2010), 2277-2300.

\bibitem{SS} F. Schulze and M. Simon, {\em Expanding solitons with non-negative curvature operator coming out of cones}. Math. Zeit., 275(1-2) (2013), 625-639.


\bibitem{WWW} J.Y. Wu, P. Wu and W. Wylie,   {\em Gradient shrinking Ricci solitons of half harmonic Weyl curvature},
http://arxiv.org/pdf/1410.7303.pdf (2014).


\bibitem{Z} Z.H. Zhang, {\em Gradient shrinking solitons with vanishing Weyl tensor}, Pacific J. Math. 242 no. 1 (2009), 189-200.

\bibitem{Z2} Z.H. Zhang, {\em On the completeness of gradient Ricci solitons}, Proc. Amer. Math. Soc. 137 no. 8 (2009) 2755-2759.

\end{thebibliography}
\end{document}